\documentclass[a4paper,12pt,reqno]{amsart}

\usepackage{amsmath}
\usepackage{amsthm}
\usepackage{amssymb}
\usepackage{latexsym}

\usepackage{graphicx}

\numberwithin{equation}{section}

\newtheorem{thm}{Theorem}[section]
\newtheorem{prop}[thm]{Proposition}
\newtheorem{lem}[thm]{Lemma}
\newtheorem{cor}[thm]{Corollary}

\theoremstyle{definition}
\newtheorem{defn}[thm]{Definition}
\newtheorem{exm}[thm]{Example}

\theoremstyle{remark}
\newtheorem{rem}[thm]{Remark}

\allowdisplaybreaks[4]

\newcommand{\al}{\alpha}
\newcommand{\be}{\beta}
\newcommand{\ga}{\gamma}
\newcommand{\Ga}{\Gamma}
\newcommand{\de}{\delta}

\newcommand{\e}{\varepsilon}

\newcommand{\la}{\lambda}

\newcommand{\sgm}{\sigma}

\renewcommand{\th}{\theta}
\newcommand{\p}{\partial}

\newcommand{\I}{\infty}
\newcommand{\Sc}[1]{\mathcal{#1}}
\newcommand{\F}{\Sc{F}}

\newcommand{\Bo}[1]{\mathbb{#1}}
\newcommand{\R}{\Bo{R}}
\newcommand{\T}{\Bo{T}}

\newcommand{\lec}{\lesssim}
\newcommand{\gec}{\gtrsim}
\newcommand{\hhat}{\widehat}
\newcommand{\bbar}{\overline}
\newcommand{\ti}{\widetilde}

\newcommand{\supp}[1]{\operatorname{supp}\> #1}
\newcommand{\Supp}[2]{\supp{#1}\subset #2}
\newcommand{\shugo}[1]{\{ #1\}}
\newcommand{\Shugo}[2]{\big\{ \, #1 \, \big| \, #2 \, \big\}}
\newcommand{\LR}[1]{{\langle #1 \rangle }}

\newcommand{\eq}[2]{\begin{equation} \label{#1} \begin{split} #2 \end{split} \end{equation}}
\newcommand{\eqq}[1]{\begin{equation*} \begin{split} #1 \end{split} \end{equation*}}
\newcommand{\mat}[1]{\begin{smallmatrix} #1 \end{smallmatrix}}
\newcommand{\norm}[2]{\big\| #1 \big\| _{#2}}
\newcommand{\tnorm}[2]{\| #1 \| _{#2}}
\newcommand{\hx}{\hspace{10pt}}

\newcommand{\eqs}[1]{\begin{gather*} #1 \end{gather*}}

\title[Norm inflation for NLS]{A remark on norm inflation for nonlinear Schr\"odinger equations}
\author[N. Kishimoto]{Nobu Kishimoto}
\email{nobu@kurims.kyoto-u.ac.jp}

\begin{document}

\begin{abstract}
We consider semilinear Schr\"odinger equations with nonlinearity that is a polynomial in the unknown function and its complex conjugate, on $\mathbb{R}^d$ or on the torus.
Norm inflation (ill-posedness) of the associated initial value problem is proved in Sobolev spaces of negative indices.
To this end, we apply the argument of Iwabuchi and Ogawa (2012), who treated quadratic nonlinearities% and utilized the power series expansion of solutions in modulation spaces and the high-to-low frequency cascade in the first nonlinear term of the series to obtain explicitly a lower bound of the solution
.
This method can be applied whether the spatial domain is non-periodic or periodic and whether the nonlinearity is gauge/scale-invariant or not.
\end{abstract}

\maketitle

%%%%%%%%%%%%%%%%%%%%%
%%%%%%%%%%%%%%%%%%%%%
%%%%%%%%%%%%%%%%%%%%%

\section{Introduction}

We consider the initial value problem for semilinear Schr\"odinger equations:
\begin{equation}\label{NLS'}
\left\{
\begin{array}{@{\,}r@{\;}l}
i\p _tu+\Delta u&=F (u,\bar{u}),\qquad (t,x)\in [0,T] \times Z,\\
u(0,x)&=\phi (x),
\end{array}
\right.
\end{equation}
where the spatial domain $Z$ is of the form $Z=\R^{d_1}\times \T ^{d_2}$, $d_1+d_2=d$, and $F (u,\bar{u})$ is a polynomial in $u,\bar{u}$ without constant and linear terms, explicitly given by
\eqq{F (u,\bar{u})=\sum _{j=1}^n\nu _ju^{q_j}\bar{u}^{p_j-q_j}}
with mutually different indices $(p_1,q_1),\dots ,(p_n,q_n)$ satisfying $p_j\ge 2$, $0\le q_j\le p_j$ and non-zero complex constants $\nu _1,\dots ,\nu _n$.

The aim of this article is to prove \emph{norm inflation} for the initial value problem \eqref{NLS'} in some negative Sobolev spaces.
We say norm inflation in $H^s(Z)$ (``\emph{NI$_s$}'' for short) occurs if for any $\de >0$ there exist $\phi \in H^\I$ and $T>0$ satisfying 
\eqq{\tnorm{\phi}{H^s}<\de ,\qquad 0<T<\de}
such that the corresponding smooth solution $u$ to \eqref{NLS'} exists on $[0,T]$ and 
\eqq{\tnorm{u(T)}{H^s}>\de ^{-1}.}
Clearly, NI$_s$ implies the discontinuity of the solution map $\phi \mapsto u$ (which is uniquely defined for smooth $\phi$ locally in time) at the origin in the $H^s$ topology, and hence the ill-posedness of \eqref{NLS'} in $H^s$.
However, NI$_s$ is a stronger instability property of the flow than the discontinuity, which only requires $0<T\lec 1$ and $\tnorm{u(T)}{H^s}\gec 1$.

Let us begin with the case of single-term nonlinearity:
\begin{equation}\label{NLS}
\left\{
\begin{array}{@{\,}r@{\;}l}
i\p _tu+\Delta u&=\nu u^q\bar{u}^{p-q},\qquad (t,x)\in [0,T] \times Z,\\
u(0,x)&=\phi (x),
\end{array}
\right.
\end{equation}
where $p\ge 2$ and $0\le q\le p$ are integers, $\nu \in \Bo{C}\setminus \{0\}$ is a constant.
The equation is invariant under the scaling transformation $u(t,x)\mapsto \la ^{\frac{2}{p-1}}u(\la ^2t,\la x)$ ($\la >0$), and the critical Sobolev index $s$ for which $\tnorm{\la ^{\frac{2}{p-1}}\phi (\la \cdot )}{\dot{H}^{s}}=\tnorm{\phi}{\dot{H}^{s}}$ is given by
\eqq{s=s_c(d,p):=\tfrac{d}{2}-\tfrac{2}{p-1}.}
The scaling heuristics suggests that the flow becomes unstable in $H^s$ for $s<s_c(d,p)$.
In addition, we will demonstrate norm inflation phenomena by tracking the transfer of energy from high to low frequencies (that is called ``high-to-low frequency cascade''), which naturally restrict us to negative Sobolev spaces.
In fact, we will show NI$_s$ with any $s<\min \shugo{s_c(d,p),0}$ for any $Z$ and $(p,q)$, as well as with some negative but scale-subcritical regularities for specific nonlinearities.
Precisely, our result reads as follows:
\begin{thm}\label{thm:main0}
Let $Z$ be a spatial domain of the form $\R ^{d_1}\times \T ^{d_2}$ with $d_1+d_2=d\ge 1$, and let $p\ge 2$, $0\le q\le p$ be integers.
Then, the initial value problem \eqref{NLS} exhibits NI$_s$ in the following cases:
\begin{enumerate}
\item $Z$ and $(p,q)$ are arbitrary, $s<\min \shugo{s_c(d,p),0}$.
\item $d,p,s$ satisfy $s=s_c(d,p)=-\frac{d}{2}$; that is, $(d,p,s)=(1,3,-\frac{1}{2})$ and $(2,2,-1)$.
\item $d=1$, $(p,q)=(2,0),(2,2)$ and $s<-1$.
\item $Z=\R^d$ with $1\le d\le 3$, $(p,q)=(2,1)$ and $s<-\frac{1}{4}$.
\item $Z=\R^{d_1}\times \T ^{d_2}$ with $d_1+d_2\le 3$, $d_2\ge 1$, $(p,q)=(2,1)$ and $s<0$.
\item $Z=\T$, $(p,q)=(4,1),(4,2),(4,3)$ and $s<0$.
\end{enumerate}
\end{thm}

There is an extensive literature on the ill-posedness of nonlinear Schr\"odinger equations, and a part of the above theorem has been proved in previous works.

Concerning ill-posedness in the sense of norm inflation, Christ, Colliander, and Tao \cite{CCT03p-1} treated the case of gauge-invariant power-type nonlinearities $\pm |u|^{p-1}u$ on $\R^d$ and proved NI$_s$ when $0<s<s_c(d,p)$ or $s\le -\frac{d}{2}$ (with some additional restriction on $s$ if $p$ is not an odd integer).
For the remaining range of regularities $-\frac{d}{2}<s<0$ (when $s_c\ge 0$) they proved the failure of uniform continuity of the solution map.
Note that this milder form of ill-posedness is not necessarily incompatible with well-posedness in the sense of Hadamard, for which continuity of the solution map is required.
Moreover, since their argument is based on scaling consideration and some ODE analysis, it does not apply in any obvious way to the cases of periodic domains,%
\footnote{One can still adapt their idea to the periodic setting with additional care.
Moreover, although their original argument did not apply to the 1d cubic case with the scaling critical regularity $s=-\frac{1}{2}$, one can modify the argument to cover that case.
See \cite{OW15p} for details.} 
non gauge-invariant nonlinearities, and complex coefficients.
Later, Carles, Dumas, and Sparber~\cite{CDS12} and Carles and Kappeler \cite{CK17} studied norm inflation in Sobolev spaces of negative indices for the problem with smooth nonlinearities (i.e., $\pm |u|^{p-1}u$ with an odd integer $p\ge 3$) in $\R^d$ and in $\T^d$, respectively.
They used a geometric optics approach to obtain NI$_s$ for $d\ge 2$ and $s<-\frac{1}{p}$ in the $\R^d$ case%
\footnote{%
In \cite{CDS12} they also proved norm inflation for generalized nonlinear Schr\"odinger equations and the Davey-Stewartson system including non-elliptic Laplacian.} 
and for $s<0$ in the $\T^d$ case with the exception of $(d,p)=(1,3)$ for which $s<-\frac{2}{3}$ was assumed.
(See \cite{C07,AC09} for related ill-posedness results.)
In fact, they showed stronger instability property than NI$_s$ for these cases; that is, norm inflation \emph{with infinite loss of regularity} (see Proposition~\ref{prop:niilr} below for the definition).
Our argument, which evaluates each term in the power series expansion of the solution directly, is different from the aforementioned works.
Note that, for smooth nonlinearities, Theorem~\ref{thm:main0} covers all the remaining cases in the range $s<\min \shugo{s_c(d,p),0}$ and extends the result to the (partially) periodic setting as well as to the case of general nonlinearities with complex coefficients.
Moreover, %as we will see in Appendix~\ref{sec:niilr}, 
our argument also gives another proof of the results in \cite{CDS12,CK17} on NI$_s$ with infinite loss of regularity; see Proposition~\ref{prop:niilr} for the precise statement.

The one-dimensional cubic equation with nonlinearity $\pm |u|^2u$ has been attracting particular attention due to its various physical backgrounds and complete integrability.
Note also that this is the only $L^2$-subcritical case among smooth and gauge-invariant nonlinearities.
In spite of the $L^2$ subcriticality, the equation becomes unstable below $L^2$ due to the Galilean invariance, both in $\R$ and in $\T$.
In fact, the initial value problem was shown to be globally well-posed in $L^2$ \cite{T87,B93-1}, whereas it was shown in \cite{KPV01,CCT03} for $\R$ and in \cite{BGT02,CCT03} for $\T$ that the solution map fails to be uniformly continuous below $L^2$.
Ill-posedness below $L^2(\T )$ was established in the periodic case by the lack of continuity of the solution map \cite{CCT03p-2,M09} and by the non-existence of solutions \cite{GO18}.
Nevertheless, one can show a priori bound in some Sobolev spaces below $L^2$ \cite{KT07,CCT08,KT12,GO18}, which prevents norm inflation.
Recent results in \cite{KT16p,KVZ17p} finally gave a priori bound on $H^s$ for $s>-\frac{1}{2}$, both in $\R$ and in $\T$.
We remark that NI$_s$ at $s=-\frac{1}{2}$ shown in Theorem~\ref{thm:main0} ensures the optimality of these results.%
\footnote{The one-dimensional cubic problem was not treated in the first version of this article.
We would like to thank T.~Oh for drawing our attention to this case. 
}
In \cite[Theorem~4.7]{KVZ17p}, Killip, Vi\c{s}an and Zhang also derived a priori bound of the solutions in the norm which is logarithmically stronger than the critical $H^{-\frac{1}{2}}$.
Motivated by this result, in addition to Theorem~\ref{thm:main0} (ii) we also show norm inflation for the one-dimensional cubic equation in some ``logarithmically subcritical'' spaces; see Proposition~\ref{prop:A} below.

Since the work of Kenig, Ponce, and Vega \cite{KPV96-NLS}, non gauge-invariant nonlinearities have also been intensively studied.
In \cite{BT06}, Bejenaru and Tao proposed an abstract framework for proving ill-posedness in the sense of discontinuity of the solution map.
They considered the quadratic NLS \eqref{NLS} on $\R$ with nonlinearity $u^2$ and obtained a complete dichotomy of Sobolev index $s$ into locally well-posed ($s\ge -1$) and ill-posed ($s<-1$) in the sense mentioned above.
Their argument is based on the power series expansion of the solution, and they proved ill-posedness by observing that high-to-low frequency cascades break the continuity of the first nonlinear term in the series.
A similar dichotomy was shown for other quadratic nonlinearities $\bar{u}^2$, $u\bar{u}$ in \cite{K09,KT10} by employing the idea of \cite{BT06}.

Later, Iwabuchi and Ogawa \cite{IO15} considered the nonlinearity $u^2$, $\bar{u}^2$ in $\R$, $\R^2$ and refined the idea of \cite{BT06} to prove ill-posedness in the sense of NI$_s$ for $s<-1$ in $\R$ and $s\le -1$ in $\R^2$.
In particular, in the two-dimensional case they could complement the local well-posedness result in $H^s(\R ^2)$, $s>-1$, which had been obtained in \cite{K09}.
% The main novelty of their argument is to use the modulation norms (i.e., a different scale than Sobolev) to estimate each term in the series.
Note that the original argument of \cite{BT06} is not likely to yield norm inflation phenomena nor discontinuity of the solution map at the threshold regularity such as $s=-1$ in the above $\R^2$ case.
We will have more discussion on this issue in the next section.
Another quadratic nonlinearity $u\bar{u}$ was investigated by the same method in \cite{IU15}, where for $\R^d$ with $d=1,2,3$ they proved norm inflation in Besov spaces $B^{-1/4}_{2,\sgm}$ of regularity $-\frac{1}{4}$ with $4<\sgm \le \I$.%
\footnote{Essentially, they also proved NI$_s$ for $s<-\frac{1}{4}$, i.e., the case (iv) of our Theorem~\ref{thm:main0}. 
}

It turns out that the method of Iwabuchi and Ogawa \cite{IO15} proving norm inflation has a wide applicability.
The purpose of the present article is to apply this method to NLS with general nonlinearities.
In the last few years the method has been used to a wide range of equations; see for instance \cite{MO15,MO16,HMO16,CP16,Ok17}.%
\footnote{%
In the first version of this article, we only considered gauge-invariant smooth nonlinearities $\nu |u|^{2k}u$, $k\in \Bo{Z}_{>0}$ and linear combinations of them.
Note, however, that the method of Iwabuchi and Ogawa \cite{IO15} had been applied before only to quadratic nonlinearities and it was the first result dealing with nonlinearities of general degrees in a unified manner.
The authors of \cite{CP16,O17} informed us that their proofs of norm inflation results followed the argument in the first version of this article.
We also remark that an estimate proved in the first version (Lemma~\ref{lem:a_k} below) was employed later in \cite{MO16,HMO16,Ok17}.
}
In \cite{O17,Ok17}, norm inflation based at general initial data was proved for NLS and some other equations.%
\footnote{%
In \cite{Ok17} non gauge-invariant nonlinearities were first treated in a general setting.
In fact Theorem~\ref{thm:main0} follows as a corollary of \cite[Proposition~2.5 and Corollary~2.10]{Ok17}.
However, we decide to include the non gauge-invariant cases in the present version in order to state Theorem~\ref{thm:main} (for multi-term nonlinearities) with more generality.
}

We make some additional remarks on Theorem~\ref{thm:main0}.
\begin{rem}
(i) Concerning one-dimensional periodic cubic NLS below $L^2$, the renormalized (or Wick ordered) equation
\eqq{i\p _tu+\p _x^2u=\pm \big( |u|^2-2-\hspace{-13pt}\int _{\T}|u|^2\big) u}
is known to behave better than the original one \eqref{NLS} with nonlinearity $\pm |u|^2u$; see \cite{OS12} for a detailed discussion.
We note that our proof can be also applied to the renormalized cubic NLS.
In fact, the solutions constructed in Theorem~\ref{thm:main0} is smooth and its $L^2$ norm is conserved.
Then, a suitable gauge transformation, which does not change the $H^s$ norm at any time, gives smooth solutions to the renormalized equation that exhibit norm inflation.  

(ii) In the periodic setting, our proof does not rely on any number theoretic consideration.
Hence, it can be easily adapted to the problem on general anisotropic tori, whether rational or irrational; that is, $Z=\R ^{d_1}\times [\R^{d_2}/(\ga _1\Bo{Z})\times \cdots \times (\ga _{d_2}\Bo{Z})]$ for any $\ga _1,\dots ,\ga _{d_2}>0$.

(iii) When $Z=\R$ and $(p,q)=(4,2)$, the example in \cite[Example~5.3]{G00p} suggests that a high-to-low frequency cascade leads to instability of the solution map when $s<-\frac{1}{8}$.
However, our argument does not imply NI$_s$ for $-\frac{1}{6}\le s<-\frac{1}{8}$ so far.
\end{rem}

There are far less results on ill-posedness for multi-term nonlinearities than for \eqref{NLS}.
However, such nonlinear terms naturally appear in application. 
For instance, the nonlinearity $6u^5-4u^3$ appears in a model related to shape-memory alloys \cite{FLS87}, and $(u+2\bar{u}+u\bar{u})u$ is relevant in the study of asymptotic behavior for the Gross-Pitaevskii equation (see {e.g.}~\cite{GNT09}).
Note that norm inflation for a multi-term nonlinearity does not immediately follow from that for each nonlinear term.
Our next result concerns the equation \eqref{NLS'} of full generality:
\begin{thm}\label{thm:main}
The initial value problem \eqref{NLS'} exhibits NI$_s$ whenever $s$ satisfies the condition in Theorem~\ref{thm:main0} for at least one term $u^{q_j}\bar{u}^{p_j-q_j}$ in $F (u,\bar{u})$, except for the case where $Z=\T$ and $F (u,\bar{u})$ contains $u\bar{u}$.

When $Z=\T$ and $F (u,\bar{u})$ contains $u\bar{u}$, NI$_s$ occurs in the following cases:
\begin{enumerate}
\item $s<0$ if $F (u,\bar{u})$ has a quintic or higher term, or one of $u^3\bar{u}$, $u^2\bar{u}^2$, $u\bar{u}^3$.
\item $s<-\frac{1}{6}$ if $F (u,\bar{u})$ has $u^4$ or $\bar{u}^4$ but no other quartic or higher terms.
\item $s\le -\frac{1}{2}$ if $F (u,\bar{u})$ has a cubic term but no quartic or higher terms.
\item $s<0$ if $F (u,\bar{u})$ has no cubic or higher terms.
\end{enumerate}
\end{thm}

In the above theorem, the range of regularities is restricted when $Z=\T$ and $F (u,\bar{u})$ has $u\bar{u}$; note that the nonlinear term $u\bar{u}$ by itself leads to NI$_s$ for $s<0$ as shown in Theorem~\ref{thm:main0}.
This restriction seems unnatural and an artifact of our argument.

The rest of this article is organized as follows.
In the next section, we recall the idea of \cite{BT06}, \cite{IO15} and discuss some common features and differences between them.
Section~\ref{sec:proof0} is devoted to the proof of Theorem~\ref{thm:main0} for the single-term nonlinearities.
Then, in Section~\ref{sec:proof} we see how to treat the multi-term nonlinearities, proving Theorem~\ref{thm:main}.
In Appendices, we consider norm inflation with infinite loss of regularity in Section~\ref{sec:niilr} and inflation of various norms with the critical regularity for the one-dimensional cubic problem in Section~\ref{sec:ap}.

%%%%%%%%%%%%%%%%%%%%%%%%%%%%%

\bigskip
% \newpage
\section{Strategy for proof}\label{BT-IO}

We will use the power series expansion of the solutions to prove norm inflation.
To see the idea, let us consider the simplest case of quadratic nonlinearity $u^2$ in \eqref{NLS}.
This amounts to considering the integral equation
\eq{eq:ie}{u(t)&=e^{it\Delta}\phi -i\int _0^te^{i(t-\tau )\Delta}\big( u(\tau )\cdot u(\tau )\big) \,d\tau \\
 &=:\Sc{L}[\phi ](t)+\Sc{N}[u,u](t),\qquad t\in [0,T].}
 
We first recall the argument of Bejenaru and Tao \cite{BT06}.
By Picard's iteration, the power series $\sum _{k=1}^\I U_k[\phi ]$ with
\eqq{&U_1[\phi ]:=\Sc{L}[\phi ], \qquad U_2[\phi ]:=\Sc{N}[\Sc{L}[\phi ],\Sc{L}[\phi ]],\\
&U_3[\phi ]:=\Sc{N}[\Sc{L}[\phi ],\Sc{N}[\Sc{L}[\phi ],\Sc{L}[\phi ]]]+\Sc{N}[\Sc{N}[\Sc{L}[\phi ],\Sc{L}[\phi ]],\Sc{L}[\phi ]],\\
&\quad \vdots \\
&U_k[\phi ]:=\sum _{k_1,k_2\ge 1;\,k_1+k_2=k}\Sc{N}[U_{k_1}[\phi ],U_{k_2}[\phi ]] \qquad (k\ge 2)}
formally gives a solution to \eqref{eq:ie}.
To justify this, we basically need the linear and bilinear estimates 
\eq{est:qwp}{\norm{\Sc{L}[\phi ]}{S}\le C\norm{\phi}{D},\qquad \norm{\Sc{N}[u_1,u_2]}{S}\le C\norm{u_1}{S}\norm{u_2}{S}}
for the space of initial data $D$ and some space $S\subset C([0,T];D)$ in which we construct a solution.
In fact, they showed (roughly speaking) the following:
\begin{quote}
Assume that \eqref{est:qwp} holds with the Banach space $D$ of initial data and some Banach space $S$.
Then, 
(i) for any $k\ge 1$ the operators $U_k:D\to S$ are well-defined and satisfies $\tnorm{U_k[\phi ]}{S}\le (C\tnorm{\phi}{D})^k$, and
(ii) there exists $\e _0>0$ (depending on the constants in \eqref{est:qwp}) such that the solution map $\phi \mapsto u[\phi ]:=\sum _{k=1}^\I U_k[\phi ]$ is well-defined on $B_D(\e _0):=\Shugo{\phi \in D}{\tnorm{\phi}{D}\le \e _0}$ and gives a solution to \eqref{eq:ie}.%
% \footnote{Note that the solution map is defined only on a small ball in $D$, while each $U_k$ is defined on the whole space.
% }
\end{quote} 

Next, consider some coarser topologies on $D$ and $S$ induced by the norms $\tnorm{~}{D'}$ and $\tnorm{~}{S'}$ weaker than $\tnorm{~}{D}$ and $\tnorm{~}{S}$, respectively.
They claimed the following:
\begin{quote}
Assume further that the solution map $\phi \mapsto u[\phi ]$ given above is continuous from $(B_D(\e _0),\tnorm{~}{D'})$ (i.e., $B_D(\e _0)$ equipped with the $D'$ topology) to $(S,\tnorm{~}{S'})$.
Then, for each $k$ the operator $U_k$ is continuous from $(B_D(\e _0),\tnorm{~}{D'})$ to $(S,\tnorm{~}{S'})$.
\end{quote}
To show the continuity of $U_k$ in coarser topologies, by its homogeneity one can restrict to sufficiently small initial data.
Then, by the estimates \eqref{est:qwp}, contribution of higher order terms $\sum _{k'>k}U_{k'}[\phi ]$ can be made arbitrarily small compared to $U_k[\phi ]$.
Combining this fact with the hypothesis that $\sum _{k\ge 1}U_k[\phi ]$ is continuous, one can show the claim by an induction argument on $k$. 

Now, this claim gives a way to prove ill-posedness in coarse topologies.
Namely, one can show the discontinuity of the solution map $\phi \mapsto \sum _{k=1}^\I U_k[\phi ]$ in coarse topologies by simply establishing the discontinuity of the (more explicit) map $\phi \mapsto U_k[\phi ]$ for at least one $k$.%
\footnote{It is worth noticing that the continuity of $U_k$ from $(B_D(\e _0),\tnorm{~}{D'})$ to $(S,\tnorm{~}{S'})$ does not imply its continuity from $(D,\tnorm{~}{D'})$ to $(S,\tnorm{~}{S'})$ in general, even though $U_k$ can be defined for all functions in $D$.
By the $k$-linearity of $U_k$, the latter continuity is equivalent to the \emph{boundedness}: $\tnorm{U_k[\phi ]}{S'}\le C\tnorm{\phi}{D'}^k$.
Hence, only disproving the boundedness of $U_k$ in coarse topologies (which may imply that the solution map is not $k$ times differentiable) is not sufficient to conclude the discontinuity of the solution map.
}
We notice that this proof of ill-posedness includes evaluating higher terms by using \eqref{est:qwp}, that is, estimates (or well-posedness) in stronger topology.

Here, we observe two facts on this method.
First, it cannot yield norm inflation in coarse topologies.
This is because the image of the continuous solution map with domain $B_D(\e _0)$ is bounded in $S$, and hence it must be bounded in weaker norms.

Secondly, the `well-posedness' estimates \eqref{est:qwp} in $D,S$ and discontinuity of some $U_k$ in $D',S'$ would imply the discontinuity of $U_k$ in any `intermediate' norms $D'',S''$ satisfying
\eqq{\tnorm{\phi}{D'}\lec \tnorm{\phi}{D''}\lec \tnorm{\phi}{D}^\th \tnorm{\phi}{D'}^{1-\th},\qquad \tnorm{u}{S'}\lec \tnorm{u}{S''}\lec \tnorm{u}{S}}
for some $0<\th <1$.
In fact, if $U_k:(B_D(\e _0),\tnorm{~}{D'})\to S'$ is not continuous, there exist $\shugo{\phi _n}\subset B_D(\e _0)$ and $\phi _\I \in B_D(\e _0)$ such that $\tnorm{\phi _n-\phi _\I}{D'}\to 0$ ($n\to \I$) but $\tnorm{U_k[\phi _n]-U_k[\phi _\I ]}{S'}\gec 1$.
Since $\shugo{\phi _n}$ is bounded in $D$, this implies that $\tnorm{\phi _n-\phi _\I}{D''}\to 0$ and $\tnorm{U_k[\phi _n]-U_k[\phi _\I ]}{S''}\gec 1$.
In particular, if we work in Sobolev spaces:
\eqq{D=H^{s_0},\quad S\hookrightarrow C([0,T];H^{s_0}),\quad D'=H^{s_1},\quad S'=C([0,T];H^{s_1})\qquad  (s_0>s_1),} 
then ill-posedness in $H^{s_1}$ as a consequence of the argument in \cite{BT06} should actually yield ill-posedness in any $H^s$, $s_1\le s<s_0$, while we have \eqref{est:qwp}, i.e., well-posedness in $H^{s_0}$.
Therefore, the regularity $s_0$ in which we invoke \eqref{est:qwp} must be automatically the threshold regularity for well-/ill-posedness.
This explains why the same argument cannot be applied to the two-dimensional quadratic NLS with nonlinearity $u^2$.
In fact, as mentioned in Introduction, \eqref{est:qwp} are obtained in $D=H^s$ when $s>-1$ (with a suitable $S$) but fails if $s\le -1$ (for any $S$ continuously embedded into $C([0,T];H^s)$), and hence well-posedness at the threshold regularity is not available in this case.

We next recall Iwabuchi and Ogawa's result \cite{IO15}, which settled the aforementioned two-dimensional case.
Indeed, the argument in \cite{IO15} is similar to that of \cite{BT06} in that it exploits the power series expansion and shows that one term in the series exhibits instability and dominates all the other terms.
Now, we notice that the existence time $T>0$ is allowed to shrink for the purpose of establishing norm inflation, while in \cite{BT06} it is fixed and uniform with respect to the initial data.
The main difference of the argument in \cite{IO15} from that of \cite{BT06} is that they worked with the estimates like
\eq{est:qwp'}{\norm{\Sc{L}[\phi ]}{S_T}\le C\norm{\phi}{D},\qquad \norm{\Sc{N}[u_1,u_2]}{S_T}\le CT^\de \norm{u_1}{S_T}\norm{u_2}{S_T}}
for the data space $D$, $S_T\subset C([0,T];D)$, and $\de >0$, and consider the expansion up to different times $T$ according to the initial data.
In fact, this enables us to take a sequence of initial data which is unbounded in $D$ (but converges to $0$ in a weaker norm), and such a set of initial data actually yields unbounded sequence of solutions.
Another feature of the argument in \cite{IO15} is that higher-order terms were estimated directly in $D'$ by using properties of specific initial data they chose; in \cite{BT06} these terms were simply estimated in $D$ by \eqref{est:qwp} that hold for general functions.%
\footnote{In fact, we do not need `well-posedness in $D$', i.e., such estimates as \eqref{est:qwp'} that hold for \emph{all} functions in $D$ and $S$.
It is enough to estimate the terms $U_k[\phi ]$ just for particularly chosen initial data $\phi$.
In some problems this consideration becomes essential; see \cite{Ok17}, Theorem~1.2 and its proof. 
}
At a technical level, another novelty in \cite{IO15} is the use of modulation space $M_{2,1}$ as $D$ instead of Sobolev spaces. 
The bilinear estimate in \eqref{est:qwp'} is then straightforward thanks to the algebra property of $M_{2,1}$.

Finally, we remark that the strategies of \cite{BT06,IO15} work well in the case that the operator $U_k$ involves a significant high-to-low frequency cascade, as mentioned in \cite{BT06}.
However, the situation is different in the case of \emph{system} of equations, as there are more than one regularity indices and one cannot simply order two pairs of regularity indices; see e.g.~\cite{MO15}, where the argument of \cite{IO15} was employed to derive norm inflation from nonlinear interactions of ``high$\times$low$\to$high'' type.

%%%%%%%%%%%%%%%%%%%%%%%%%

\bigskip
% \newpage
\section{Proof of Theorem~\ref{thm:main0}}\label{sec:proof0}

Let us first consider the case of single-term nonlinearity and prove Theorem~\ref{thm:main0}.
The argument in this section basically follows that in \cite{IO15}.
Since the coefficient $\nu \neq 0$ plays no role in our proof, we assume $\nu =1$ for simplicity.
We write 
\eqq{\mu _{p,q}(z_1,\dots ,z_p):=\prod _{l=1}^qz_l\prod _{m=q+1}^p\bar{z}_m,\qquad \mu _{p,q}(z):=\mu _{p,q}(z,\dots ,z),}
so that $u^q\bar{u}^{p-q}=\mu _{p,q}(u)$.

\begin{defn}\label{defn:U_k}
For $\phi \in L^2(Z)$, we (formally) define
\eqq{U_1[\phi ](t)&:=e^{it\Delta}\phi ,\\
U_k[\phi ](t)&:=-i\sum _{\mat{k_1,\dots ,k_p\ge 1\\ k_1+\dots +k_p=k}}\int _0^t e^{i(t-\tau )\Delta}\mu _{p,q}\big( U_{k_1}[\phi ],\dots ,U_{k_p}[\phi ]\big) (\tau )\,d\tau ,\qquad k\ge 2.}
\end{defn}

Note that $U_{k}[\phi ]=0$ unless $k\equiv 1\mod p-1$.

The expansion $u=\sum _{k=1}^\I U_k[\phi ]$ of a (unique) solution $u$ to \eqref{NLS} will play a crucial role in the proof.
To make sense of this representation, we use modulation spaces.
The notion of modulation spaces was introduced by Feichtinger in the 1980s \cite{F83} and nowadays it has become one of the common tools in the study of nonlinear evolution PDEs; see e.g.~the survey \cite{RSW12} and references therein. 

\begin{defn}
Let $A>0$ be a dyadic number.
Define the space $M_A$ as the completion of $C_0^\I (Z)$ with respect to the norm
\[ \norm{f}{M_A}:=\sum _{\xi \in A\Bo{Z}^d}\norm{\hhat{f}}{L^2(\xi +Q_A)},\]
where $Q_A:=[-\frac{A}{2},\frac{A}{2})^d$.
% For $s\in \R$, we define $M_A^s:=(1-\Delta )^{-s/2}M_A$ with norm
% \eqq{\tnorm{f}{M^s_A}:=\sum _{\xi \in A\Bo{Z}^d}\norm{\LR{\xi}^s\hhat{f}}{L^2(\xi +Q_A)}.}
\end{defn}
\begin{rem}
We consider the space $M_A$ with $A<1$ only when $Z=\R ^d$.
For $Z=\R ^{d_1}\times \T ^{d_2}$, the $L^2(\xi +Q_A)$ norm in the above definition means the $L^2$ norm restricted onto $(\xi +Q_A)\cap \hhat{Z}$, where $\hhat{Z}:=\R ^{d_1}\times \Bo{Z} ^{d_2}$. 
If $Z=\T^d$, the space $M_1$ coincides with the Wiener algebra $\F L^1(\T ^d)$.
\end{rem}

We will only use the following properties of the space $M_A$.
The proof is elementary, and thus it is omitted.
\begin{lem}\label{lem:M_A}
(i) $M_A\cong _AM_1$,\hx $H^{\frac{d}{2}+\e}\hookrightarrow M_1\hookrightarrow L^2$\hx ($\e >0$).

(ii) There exists $C=C(d)>0$ such that for any $f,g\in M_A$, we have
\[ \norm{fg}{M_A}\le CA^{\frac{d}{2}}\norm{f}{M_A}\norm{g}{M_A}.\]
\end{lem}

Since the space $M_A$ is a Banach algebra and the linear propagator $e^{it\Delta}$ is unitary in $M_A$, we can easily show the following multilinear estimates. 

\begin{lem}\label{lem:U_k}
Let $A\ge 1$ be a dyadic number and $\phi \in M_A$ with $\tnorm{\phi}{M_A}\le M$.
Then, there exists $C>0$ independent of $A$ and $M$ such that
\eqq{\norm{U_k[\phi ](t)}{M_A}\le t^{\frac{k-1}{p-1}}(CA^{\frac{d}{2}}M)^{k-1}M}
for any $t\ge 0$ and $k\ge 1$.
\end{lem}

\begin{proof}
Let $\shugo{a_k} _{k=1}^\I$ be the sequence defined by
\[ a_1=1,\qquad a_k=\frac{p-1}{k-1}\sum _{\mat{k_1,\dots ,k_p\ge 1\\ k_1+\dots +k_p=k}}a_{k_1}\cdots a_{k_p}\qquad (k\ge 2).\]

As observed in \cite[Eq.~(16)]{BT06}, one can show inductively that $a_k\le C^k$ for some $C>0$.
To be more precise, we state it as the following lemma.
The $p=2$ case can be found in \cite[Lemma~4.2]{MO16} with a detailed proof.
\begin{lem}\label{lem:a_k}
Let $\shugo{b_k}_{k=1}^\I$ be a sequence of nonnegative real numbers such that 
\[ b_{k} \le C\sum _{\mat{k_1,\dots ,k_p\ge 1\\ k_1+\dots +k_p=k}}b_{k_1}\cdots b_{k_p},\qquad k\ge 2\]
for some $p\ge 2$ and $C>0$.
Then, we have
\[ b_k\le b_1C_0^{k-1},\qquad k\ge 1;\qquad C_0:=\frac{\pi ^2}{6}(Cp^2)^{\frac{1}{p-1}}b_1.\]
\end{lem}

% \begin{proof}
% It suffices to show
% \eq{claim1}{k^2a_k\le a_1C_0^{k-1},\qquad k\ge 1.}
% \eqref{claim1} is trivial if $k=1$.
% Assuming it is true for $k=1,2,\dots ,K-1$ ($K\ge 2$), we see that
% \eqq{K^2a_K&\le C\sum _{\mat{k_1,\dots ,k_p\ge 1\\ k_1+\dots +k_p=K}}a_{k_1}\cdots a_{k_p}\cdot p\sum _{j=1}^p k_j^2\le Cpa_1^pC_0^{K-p}\sum _{\mat{k_1,\dots ,k_p\ge 1\\ k_1+\dots +k_p=K}}\frac{k_1^2+\cdots +k_p^2}{k_1^2k_2^2\cdots k_p^2}\\
% &\le Cpa_1^pC_0^{K-p}\cdot p\Big( \sum _{k=1}^{K-p+1}\frac{1}{k^2}\Big) ^{p-1}\le Cp^2a_1^pC_0^{K-p}\Big( \frac{\pi ^2}{6}\Big) ^{p-1},
% }
% hence \eqref{claim1} is also true for $k=K$.
% \end{proof}

By Lemma~\ref{lem:a_k}, it holds $a_k\le C_0^{k-1}$ for some $C_0>0$.
Thus, it suffices to show
\eqq{\norm{U_{k}[\phi ](t)}{M_A}\le a_kt^{\frac{k-1}{p-1}}(C_1A^{\frac{d}{2}}M)^{k-1}M,\qquad t\ge 0,\quad k\ge 1}
for some $C_1>0$.
This is trivial if $k=1$.
Let $k\ge 2$, and assume the above estimate for $U_1,U_2,\dots ,U_{k-1}$.
Using Lemma~\ref{lem:M_A}, we have 
\eqq{\norm{U_{k}[\phi ](t)}{M_A}&\le CA^{\frac{d}{2}(p-1)}\sum _{\mat{k_1,\dots ,k_p\ge 1\\ k_1+\dots +k_p=k}}\int _0^t \prod _{j=1}^p\norm{U_{k_j}[\phi ](\tau )}{M_A}\,d\tau \\
&\le CA^{\frac{d}{2}(p-1)}(C_1A^{\frac{d}{2}}M)^{k-p}M^p\sum _{\mat{k_1,\dots ,k_p\ge 1\\ k_1+\dots +k_p=k}}a_{k_1}\cdots a_{k_p}\int _0^t \tau ^{\frac{k-p}{p-1}}\,dt\\
&=Ca_kC_1^{k-p}(A^{\frac{d}{2}}M)^{k-1}Mt^{\frac{k-1}{p-1}}.}
The estimate for $U_k$ follows by setting $C_1$ to be $C^{\frac{1}{p-1}}$ with the constant $C$ in the last line, which is independent of $k$.
\end{proof}

A standard argument (cf.~\cite[Theorem~3]{BT06}) with Lemma~\ref{lem:M_A} (ii) and Lemma~\ref{lem:U_k} shows the following local well-posedness of \eqref{NLS} in $M_A$.

\begin{cor}\label{cor:lwp}
Let $A\ge 1$ be dyadic, and $M>0$.
If $0<T\ll (A^{d/2}M)^{-(p-1)}$, then for any $\phi \in M_A$ with $\tnorm{\phi}{M_A}\le M$ the following holds.

(i) A unique solution $u$ to the integral equation associated with \eqref{NLS},
\eq{eq:ie'}{u(t)=e^{it\Delta}\phi -i\int _0^t e^{i(t-\tau )\Delta}\mu _{p,q}(u(\tau ))\,d\tau ,\qquad t\in [0,T]}
exists in $C([0,T];M_A)$.

(ii) The solution $u$ given in (i) has the expression 
\eqq{u=\sum _{k=1}^\I U_{k}[\phi ]=\sum _{l=0}^\I U_{(p-1)l+1}[\phi ],}
which converges absolutely in $C([0,T];M_A)$.
%
%(iii) If $\phi \in M^s_A$ for some $s>0$, then $u\in C([0,T];M_A^s)$ with the same $T$.
\end{cor}
\begin{proof}
(i) Let 
\eqq{\Psi _{\phi}[u](t):= e^{it\Delta}\phi -i\int _0^t e^{i(t-\tau )\Delta}\mu _{p,q}(u(\tau ))\,d\tau ,}
then from Lemma~\ref{lem:M_A} (ii) we have
\eqq{\norm{\Psi _\phi [u]}{L^\I (0,T;M_A)}\le \tnorm{\phi}{M_A}+CTA^{\frac{d}{2}(p-1)}\tnorm{u}{L^\I (0,T;M_A)}^p}
and that $\Psi$ is a contraction on a ball in $C([0,T];M_A)$ if $TA^{\frac{d}{2}(p-1)}\tnorm{\phi}{M_A}^{p-1}\ll 1$.

(ii) The series $u=\sum _{k\ge 1}U_k[\phi ]$ converges in $C([0,T];M_A)$ by virtue of Lemma~\ref{lem:U_k}.
By uniqueness, it suffices to show that $u$ solves the equation \eqref{eq:ie'}.
Let $u_K:=\sum _{k=1}^KU_k[\phi ]$, so that $u=\lim _{K\to \I}u_K$ in $C([0,T];M_A)$.
We see that $\Psi _\phi [u_K]-u_K$ consists of $k$-linear terms in $\phi$ with $K+1\le k\le pK$, and we can show 
\eqq{\tnorm{\Psi _\phi [u_K]-u_K}{L^\I (0,T;M_A)}\le C(CT^{\frac{1}{p-1}}A^{\frac{d}{2}}M)^KM}
by an argument similar to Lemma~\ref{lem:U_k}.
By letting $K\to \I$, we obtain $\Psi _\phi [u]=u$.
\end{proof}

\begin{rem}
(i) In $M_A$ we have \emph{unconditional} local well-posedness. %and persistence of regularity for \eqref{NLS}.
In particular, the embedding (Lemma~\ref{lem:M_A} (i)) shows that the unique solution with initial data in some high-regularity Sobolev space exists on a time interval $[0,T]$ %with $0<T\ll (A^{d/2}\tnorm{\phi}{M_A})^{-(p-1)}$ and belongs to $C([0,T];M_A)$, and thus 
and coincides with the solution constructed in Corollary~\ref{cor:lwp}.

(ii) %We remark that this is also true even for the $L^2$-supercritical power $p$ (i.e., $s_c(d,p)>0$).
In the following proof of Theorem~\ref{thm:main0} we will take initial data that are localized in frequency on several cubes of side length $O(A)$ located in $\shugo{|\xi |\gg \max (1,A)}$.
For such initial data the $L^2$ norm is comparable with the $M_A$ norm, but much smaller than the Sobolev norms of positive indices.
In the $L^2$-supercritical cases (i.e., $s_c(d,p)>0$), no reasonable well-posedness is expected in $L^2$, while the use of higher Sobolev space would verify the power series expansion only on a smaller time interval. 
In this regard, the space $M_A$ is suitable for our purpose.
% For some initial data compactly supported in frequency, we can still construct power series $L^2$ solutions \eqref{powerseries} with $T=T(\tnorm{\phi}{L^2})$.
% However, it is not trivial that these $L^2$ solutions will remain smooth and coincide with the usual smooth ones in the whole time interval $[0,T]$.
\end{rem}

% \newpage
Let $N,A$ be dyadic numbers to be specified so that $N\gg 1$ and $0< A\ll N$ ($1\le A\ll N$ when $Z$ has a periodic direction).
In the proof of norm inflation, we will use initial data $\phi$ of the following form:
\eq{cond:phi}{&\hhat{\phi}=rA^{-\frac{d}{2}}N^{-s}\chi _\Omega \quad \text{with a positive constant $r$ and a set $\Omega$ satisfying}\\
&\Omega = \bigcup _{\eta \in \Sigma}(\eta +Q_A)\hx \text{for some $\Sigma \subset \shugo{\xi \in \R ^d:|\xi |\sim N}$ s.t. $\# \Sigma \le 3$}.}
Note that $\tnorm{\phi}{M_A}\sim rN^{-s}$, $\tnorm{\phi}{H^s}\sim r$.

We derive Sobolev bounds of $U_k[\phi ](t)$ with $\phi$ satisfying the above condition.

\begin{lem}\label{lem:supp}
There exists $C>0$ such that for any $\phi$ satisfying \eqref{cond:phi} and $k\ge 1$, we have
\eqq{\big| \supp{\hhat{U_{k}[\phi]}(t)}\big| \le C^kA^d,\qquad t\ge 0.}
\end{lem}

\begin{proof}
Since the $\xi$-support of $\hhat{U_{k}[\phi]}$ is determined by a spatial convolution of $k$ copies of $\hat{\phi}$ or $\hat{\bar{\phi}}=\bbar{\hat{\phi}(-\cdot )}$, it is easily seen that 
\eqq{\Supp{\hhat{U_{k}[\phi]}(t)}{\bigcup _{\eta \in \Sc{S}_k}\big( \eta +Q_{kA}\big)}}
for all $t\ge 0$, where $\Sc{S}_1:=\Sigma$ and 
\eqq{\Sc{S}_k:=&\Shugo{\eta \in \R^d}{\eta =\sum _{l=1}^k\eta _l,\,\eta _l\in \Sigma \cup (-\Sigma )\;(1\le l\le k)},\qquad k\ge 2.}
Since $\# \Sc{S}_k\le 6^k$, we have
\eqq{\big| \supp{\hhat{U_{k}[\phi]}(t)}\big| \le \big| Q_{kA}\big| \# \Sc{S}_k\le (kA)^d6^k\le C^kA^d.\qedhere}
\end{proof}

\begin{lem}\label{lem:U_k_H^s}
Let $\phi$ satisfy \eqref{cond:phi}.
Assume that $s<0$.
Then, there exists $C>0$ depending only on $d,p,s$ such that the following holds.
\begin{enumerate}
\item $\norm{U_1[\phi ](T)}{H^s}\le Cr$\hx for any $T\ge 0$.
\item $\norm{U_{k}[\phi](T)}{H^s}\le Cr(C\rho)^{k-1}A^{-\frac{d}{2}}N^{-s}f_s(A)$\hx for any $T\ge 0$ and $k\ge 2$, where
\eqq{\rho :=rA^{\frac{d}{2}}N^{-s}T^{\frac{1}{p-1}},\qquad f_s(A):=\norm{\LR{\xi}^s}{L^2(\shugo{|\xi |\le A})}.}
\end{enumerate}
\end{lem}

\begin{proof}
(i) is easily verified.
For (ii), we see that
\eqq{&\norm{U_{k}[\phi](t)}{H^s}\le \norm{\LR{\xi}^s}{L^2(\supp{\hhat{U_{k}[\phi]}(t)})}\sup _{\xi \in \R^d}\big| \hhat{U_k[\phi]}(t,\xi )\big| \\
&\le \norm{\LR{\xi}^s}{L^2(\supp{\hhat{U_{k}[\phi]}(t)})}\sum _{\mat{k_1,\dots ,k_p\ge 1\\ k_1+\dots +k_p=k}}\int _0^t \norm{\big| v_{k_1}(\tau )\big| *\cdots *\big| v_{k_p}(\tau )\big|}{L^\I}\,d\tau ,}
where $v_{k_l}$ is either $\hhat{U_{k_l}[\phi]}$ or $\hhat{\bbar{U_{k_l}[\phi ]}}$.
By Young's inequality, the above is bounded by
\eqq{&\norm{\LR{\xi}^s}{L^2(\supp{\hhat{U_{k}[\phi]}(t)})}\sum _{\mat{k_1,\dots ,k_p\ge 1\\ k_1+\dots +k_p=k}}\int _0^t \norm{v_{k_1}(\tau )}{L^2}\norm{v_{k_2}(\tau )}{L^2}\prod _{l=3}^{p}\norm{v_{k_l}(\tau )}{L^1}\,d\tau \\
&\le \norm{\LR{\xi}^s}{L^2(\supp{\hhat{U_{k}[\phi]}(t)})}\sum _{\mat{k_1,\dots ,k_p\ge 1\\ k_1+\dots +k_p=k}}\int _0^t \prod _{l=3}^p\big| \supp{\hhat{U_{k_l}[\phi]}(\tau )}\big| ^{\frac{1}{2}}\prod _{l=1}^p\norm{\hhat{U_{k_l}[\phi]}(\tau )}{L^2}\,d\tau .}
Since $s<0$, for any bounded set $D\subset \R ^d$ it holds that
\eqq{\big| \shugo{\LR{\xi}^{s}>\la}\cap D\big| \le \big| \shugo{\LR{\xi}^s>\la}\cap B_D\big| \qquad (\la >0),}
where $B_D\subset \R ^d$ is the ball centered at the origin with $|D|=|B_D|$.
This implies that $\tnorm{\LR{\xi}^s}{L^2(D)}\le \tnorm{\LR{\xi}^s}{L^2(B_D)}$.
Moreover, it follows from Lemma~\ref{lem:U_k} with $M=CrN^{-s}$ that
\eqq{\norm{U_{k}[\phi](t)}{L^2}\le \norm{U_{k}[\phi](t)}{M_A}\le Ct^{\frac{k-1}{p-1}}(CrA^{\frac{d}{2}}N^{-s})^{k-1}rN^{-s},\qquad k\ge 1 .}
Hence, we apply Lemma~\ref{lem:supp} to bound the above by
\eqq{&\norm{\LR{\xi}^s}{L^2(\shugo{|\xi |\le C^{\frac{k}{d}}A})}\cdot C^{\frac{k}{2}}A^{\frac{d(p-2)}{2}}\sum _{\mat{k_1,\dots ,k_p\ge 1\\ k_1+\dots +k_p=k}}\int _0^t \prod _{l=1}^p\big[ C\tau ^{\frac{k_l-1}{p-1}}(CrA^{\frac{d}{2}}N^{-s})^{k_l-1}rN^{-s}\big] \,d\tau \\
&\le C^k\norm{\LR{\xi}^s}{L^2(\shugo{|\xi |\le A})}A^{\frac{d(p-2)}{2}+\frac{d}{2}(k-p)}(rN^{-s})^k\int _0^t\tau ^{\frac{k-p}{p-1}}\,d\tau \\
&\le f_s(A)A^{\frac{d}{2}(k-2)}(CrN^{-s})^kt^{\frac{k-1}{p-1}},}
which is the desired one.
\end{proof}

We observe the following lower bounds on the $H^s$ norm of the first nonlinear term in the expansion of the solution.
\begin{lem}\label{lem:U_p}
The following estimates hold for any $s\in \R$.
\begin{enumerate}
\item Let $(p,q)$ and $Z=\R^{d_1}\times \T ^{d_2}$ be arbitrary.
For $1\le A\ll N$, we define the initial data $\phi$ by \eqref{cond:phi} with $\Sigma =\shugo{Ne_d, -Ne_d, 2Ne_d}$, where $e_d:=(0,\dots ,0,1)\in \R ^d$.
If $0<T\ll N^{-2}$, then we have
\eqq{\norm{U_p[\phi ](T)}{H^s}\gec r\rho ^{p-1}A^{-\frac{d}{2}}N^{-s}f_s(A).}
% \eq{est:U_p}{\norm{U_p[\phi ](T)}{H^s}\gec r\rho ^{p-1}A^{-\frac{d}{2}}N^{-s}f_s(A).}
\item Let $(p,q)=(2,1)$ and $Z=\R ^d$, $1\le d\le 3$.
For $N\gg 1$, define $\phi$ by 
\eqq{\hhat{\phi}:=rN^{\frac{1}{2}-s}\chi _{Ne_d+\ti{Q}_{N^{-1}}}\hx \text{with}\hx r>0,\hx \ti{Q}_{N^{-1}}:=[-\tfrac{1}{2},\tfrac{1}{2})^{d-1}\times [-\tfrac{1}{2N},\tfrac{1}{2N}).}
Then, for any $0<T\ll 1$ we have
\eqq{\norm{U_2[\phi ](T)}{H^s}\gec r^2N^{-2s-\frac{1}{2}}T.}
\item Let $(p,q)=(2,1)$ and $Z=\R ^{d_1}\times \T ^{d_2}$ with $d_1+d_2\le 3$, $d_2\ge 1$.
Define $\phi$ by \eqref{cond:phi} with $A=1$, $\Sigma =\shugo{Ne_d}$.
Then, for any $0<T\ll 1$ we have %\eqref{est:U_p}.
\eqq{\norm{U_2[\phi ](T)}{H^s}\gec r^2N^{-2s}T.}
\item Let $(p,q)=(4,1)$ or $(4,2)$ or $(4,3)$ and $Z=\T$.
Define $\phi$ by \eqref{cond:phi} with $A=1$, $\Sigma =\shugo{-N,2N,3N}$.
Then, for any $T>0$ we have %\eqref{est:U_p}.
\eqq{\norm{U_4[\phi ](T)}{H^s}\gec r^4N^{-4s}T.}
\end{enumerate}
\end{lem}

\begin{proof}
Note that
\eqq{\hhat{U_p[\phi ]}(T,\xi )=ce^{-iT|\xi |^2}\int _{\Ga}\prod _{l=1}^q\hhat{\phi}(\xi _l)\prod _{m=q+1}^p\bbar{\hhat{\phi}(\xi _m)}\int _0^T e^{it\Phi}\,dt,}
where 
\eqs{\Ga :=\Shugo{(\xi _1,\dots ,\xi _p)}{\sum _{l=1}^q\xi _l-\sum _{m=q+1}^p\xi _m=\xi},\quad 
\Phi :=|\xi |^2-\sum _{l=1}^q|\xi _l|^2+\sum _{m=q+1}^p|\xi _m|^2.}

(i) If we restrict $\xi$ to $Q_A$, we have
\eqq{\hhat{U_p[\phi ]}(T,\xi )=c(rA^{-\frac{d}{2}}N^{-s})^pe^{-iT|\xi |^2}\sum _{(\eta _1,\dots ,\eta _p)}\int _{\Ga}\prod _{l=1}^p\chi _{\eta _l+Q_A}(\xi _l)\int _0^T e^{it\Phi}\,dt,}
where the sum is taken over the set 
\eqq{\Shugo{(\eta _1,\dots ,\eta _p)\in \shugo{\pm Ne_d ,2Ne_d}^p}{\sum _{l=1}^q\eta _l-\sum _{m=q+1}^p\eta _m=0},} which is non-empty for any $(p,q)$.%
\footnote{If $p$ is even, we can choose $\eta _l$ to be $Ne_d$ or $-Ne_d$ so that $\sum _{l=1}^q\eta _l-\sum _{m=q+1}^p\eta _m=0$.
If $p$ is odd, we choose $\eta _1=2Ne_d$ and $\eta _2$ to be $Ne_d$ or $-Ne_d$ so that the output from these two frequencies is either $Ne_d$ or $-Ne_d$.
Then, the other $\eta _j$ can be chosen as for $p$ even.
}
Since $|\Phi| \lec N^2$ in the integral, for $0<T\ll N^{-2}$ we have
\eqq{|\hhat{U_p[\phi ]}(T,\xi )|\gec (rA^{-\frac{d}{2}}N^{-s})^p(A^d)^{p-1}T\chi _{p^{-1}Q_{A}}(\xi ),}
and thus
\eqq{\norm{U_p[\phi ](T)}{H^s}\gec (rA^{-\frac{d}{2}}N^{-s})^p(A^d)^{p-1}T\norm{\LR{\xi}^s}{L^2(p^{-1}Q_{A})}\sim r\rho ^{p-1}A^{-\frac{d}{2}}N^{-s}f_s(A).}

(ii) In this case we have
\eqq{\hhat{U_2[\phi ]}(T,\xi )=c(rN^{\frac{1}{2}-s})^2e^{-iT|\xi |^2}\int _{\xi _1-\xi _2=\xi}\chi _{\ti{Q}_{N^{-1}}}(\xi _1-Ne_d)\chi _{\ti{Q}_{N^{-1}}}(\xi _2-Ne_d)\int _0^T e^{it\Phi}\,dt,}
and in the integral, for $\xi =\xi _1-\xi _2\in \ti{Q}_{N^{-1}}$, 
\eqq{\Phi =|\xi |^2-|\xi _1|^2+|\xi _2|^2=|\xi |^2-|\xi _1-Ne_d|^2+|\xi _2-Ne_d|^2-2(\xi _1-\xi _2)\cdot Ne_d=O(1).}
Hence, if $0<T\ll 1$, we have
\eqq{|\hhat{U_2[\phi ]}(T)|\gec (rN^{\frac{1}{2}-s})^2N^{-1}T\chi _{2^{-1}\ti{Q}_{N^{-1}}},\qquad \norm{U_2[\phi ](T)}{H^s}\gec (rN^{\frac{1}{2}-s})^2N^{-\frac{3}{2}}T}
for any $s\in \R$.

(iii) Similarly to (ii), we see that
\eqq{\hhat{U_2[\phi ]}(T,(\xi ',0))=c(rN^{-s})^2e^{-iT|\xi |^2}\int _{\xi _1'-\xi _2'=\xi '}\chi _{[-1/2,1/2 )^{d-1}}(\xi _1')\chi _{[-1/2,1/2 )^{d-1}}(\xi _2')\int _0^T e^{it\Phi}\,dt,}
where the integral in $\xi '=(\xi _1,\dots ,\xi _{d-1})$ vanishes if $Z=\T$.
In the integral,
\eqq{\Phi =|(\xi ',0)|^2-|(\xi _1',N)|^2+|(\xi _2',N)|^2=O(1).}
Hence, if $0<T\ll 1$, we have
\eqq{\norm{U_2[\phi ](T)}{H^s}\ge \norm{\LR{\xi}^s\hhat{U_2[\phi ]}(T)}{L^2(Q_{1/2})}\gec (rN^{-s})^2T}
for any $s\in \R$.

(iv) We first consider $(p,q)=(4,1)$; the case of $(4,3)$ is treated in the same way.
Observe that
\eqq{&\Shugo{(\eta _1,\dots ,\eta _4)\in \shugo{-N,2N,3N}^4}{\eta _1-\eta _2-\eta _3-\eta _4=0}\\
&=\shugo{(3N,-N,2N,2N),\,(3N,2N,-N,2N),\,(3N,2N,2N,-N)}.}
Therefore, we have
\eqq{\hhat{U_4[\phi ]}(T,0)&=c(rN^{-s})^4\sum _{\mat{\xi _1,\dots ,\xi _4\in \Bo{Z}\\ \xi _1-\xi _2-\xi _3-\xi _4=0}}\prod _{l=1}^4\chi _{\shugo{-N,2N,3N}}(\xi _l)\int _0^T e^{it\Phi}\,dt\\
&=3c(rN^{-s})^4\int _0^T e^{it\{ 0^2-(3N)^2+(-N)^2+(2N)^2+(2N)^2\}}\,dt =3c(rN^{-s})^4T,}
which implies 
\eqq{\norm{U_4[\phi ](T)}{H^s}\gec (rN^{-s})^4T}
for any $s\in \R$ and $T>0$.

Next, we consider $(p,q)=(4,2)$, which is very similar to the above.
Since
\eqq{&\Shugo{(\eta _1,\dots ,\eta _4)\in \shugo{-N,2N,3N}^4}{\eta _1+\eta _2-\eta _3-\eta _4=0}\\
&=\Shugo{(\eta _1,\dots ,\eta _4)\in \shugo{-N,2N,3N}^4}{\shugo{\eta _1,\eta _2}=\shugo{\eta _3,\eta _4}},}
we have
\eqq{\hhat{U_4[\phi ]}(T,0)&=c(rN^{-s})^4\sum _{\mat{\xi _1,\dots ,\xi _4\in \Bo{Z}\\ \xi _1+\xi _2-\xi _3-\xi _4=0}}\prod _{l=1}^4\chi _{\shugo{-N,2N,3N}}(\xi _l)\int _0^T e^{it\Phi}\,dt=15c(rN^{-s})^4T,}
and the same estimate holds.
\end{proof}

Now, we are in a position to prove norm inflation.

\begin{proof}[Proof of Theorem~\ref{thm:main0}]
We first recall that $U_k[\phi]=0$ unless $k\equiv 1\mod p-1$.
If the initial data $\phi$ satisfies \eqref{cond:phi}, Corollary~\ref{cor:lwp} guarantees existence of the solution to \eqref{NLS} and the power series expansion in $M_A$ up to time $T$ whenever $\rho =rA^{\frac{d}{2}}N^{-s}T^{\frac{1}{p-1}}\ll 1$.

\underline{Case 1}: General $Z$ and $(p,q)$, $s<\min \shugo{s_c(d,p),0}$.

Take $\phi$ as in Lemma~\ref{lem:U_p} (i).
From Lemmas~\ref{lem:U_k_H^s} and \ref{lem:U_p}, under the conditions 
\eq{cond:1}{T\ll N^{-2},\quad \rho \ll 1,\quad r\rho ^{p-1}A^{-\frac{d}{2}}N^{-s}f_s(A)\gg r,}
we have
\eqq{\norm{u(T)}{H^s}\sim \norm{U_p[\phi ](T)}{H^s}\sim r\rho ^{p-1}A^{-\frac{d}{2}}N^{-s}f_s(A).}
Now, we set
\eqq{r=(\log N)^{-1},\quad  A\sim (\log N)^{-\frac{p+1}{|s|}}N,\quad T=(A^{-\frac{d}{2}}N^s)^{p-1},}
so that $\rho = (\log N)^{-1}\ll 1$.
The super-critical assumption $s<s_c(d,p)=\frac{d}{2}-\frac{2}{p-1}$ ensures that
\eqq{T\sim (\log N)^{\frac{d(p+1)}{2|s|}(p-1)}N^{(s-\frac{d}{2})(p-1)}\ll N^{-2}.}
Moreover, since $f_s(A)\gec A^{\frac{d}{2}+s}$ for any $s<0$ and $A\ge 1$, we see that
\eqq{r\rho ^{p-1}A^{-\frac{d}{2}}N^{-s}f_s(A)\gec r\rho ^{p-1}A^{s}N^{-s}\sim \log N\gg (\log N)^{-1}=r.}
Therefore, \eqref{cond:1} is fulfilled and we have $\tnorm{u(T)}{H^s}\gec \log N$.
Noticing $\tnorm{\phi}{H^s}\sim r=(\log N)^{-1}$ and $T\ll N^{-2}$, we show norm inflation by letting $N\to \I$.

\underline{Case 2}: $Z=\R$ or $\T$, $(p,q)=(2,0)$ or $(2,2)$, $-\frac{3}{2}\le s<-1$.

We take the same initial data $\phi$ as in Case 1, but with
\eqq{r=(\log N)^{-1},\quad A=1,\quad T=(\log N)^{-1}N^{-2}.}
Then, $T\ll N^{-2}$, $\rho =(\log N)^{-2}N^{-2-s}\ll 1$ by $s\ge -\frac{3}{2}$ and
\eqq{r\rho ^{p-1}A^{-\frac{d}{2}}N^{-s}f_s(A)\sim r\rho N^{-s}= (\log N)^{-3}N^{-2-2s}\gg 1\gg r}
by $s<-1$.
Hence, \eqref{cond:1} holds and we have $\tnorm{u(T)}{H^{s}}\sim (\log N)^{-3}N^{-2-2s}\gg 1$, which together with $\tnorm{\phi}{H^s}\sim r\ll 1$ and $T\ll 1$ shows norm inflation by taking $N$ large.

\underline{Case 3}: $Z=\R$ or $\T$, $p=3$, $s=-\frac{1}{2}$.

Take the same $\phi$ as in Case 1, but with
\eqq{r=(\log N)^{-\frac{1}{12}},\quad  A\sim (\log N)^{-\frac{1}{4}}N,\quad T=(\log N)^{-\frac{1}{12}}N^{-2}.}
Then, $T\ll N^{-2}$, $\rho \sim (\log N)^{-\frac{1}{4}}\ll 1$ and
\eqq{r\rho ^{p-1}A^{-\frac{d}{2}}N^{-s}f_s(A)\sim r\rho ^2A^{-\frac{1}{2}}N^{\frac{1}{2}}(\log A)^{\frac{1}{2}}\sim (\log N)^{\frac{1}{24}}\gg 1\gg r.}
Hence, \eqref{cond:1} holds and we have $\tnorm{u(T)}{H^{-\frac{1}{2}}}\sim (\log N)^{\frac{1}{24}}\gg 1$, which implies norm inflation as well.

\underline{Case 4}: $Z=\R ^2$ or $\R \times \T$ or $\T^2$, $(p,q)=(2,0)$ or $(2,2)$, $s=-1$.

We follow the argument in Case 1 again, but with
\eqq{r=(\log N)^{-\frac{1}{12}},\quad  A\sim (\log N)^{-\frac{1}{4}}N,\quad T=(\log N)^{-\frac{1}{6}}N^{-2}.}
Then, $T\ll N^{-2}$, $\rho \sim (\log N)^{-\frac{1}{2}}\ll 1$ and
\eqq{r\rho ^{p-1}A^{-\frac{d}{2}}N^{-s}f_s(A)\sim r\rho A^{-1}N(\log A)^{\frac{1}{2}}\sim (\log N)^{\frac{1}{6}}\gg 1\gg r.}
Hence, \eqref{cond:1} holds and we have $\tnorm{u(T)}{H^{-1}}\sim (\log N)^{\frac{1}{6}}\gg 1$, which shows NI$_{-1}$.

\underline{Case 5}: $Z=\R ^{d_1}\times \T ^{d_2}$ with $d_1+d_2\le 3$, $d_2\ge 1$, $(p,q)=(2,1)$, and $\frac{d}{2}-2\le s<0$.

Take $\phi$ as in Lemma~\ref{lem:U_p} (iii) and choose $r,T$ as $r=(\log N)^{-1}$ and $T=N^s$, which implies
\eqq{T\ll 1,\quad \rho \sim rN^{-s}T=(\log N)^{-1}\ll 1,\quad r\rho N^{-s}\sim (\log N)^{-2}N^{-s}\gg 1\gg r.}
From Lemmas~\ref{lem:U_k_H^s} and \ref{lem:U_p}, we have $\tnorm{u(T)}{H^s}\sim \norm{U_2[\phi ](T)}{H^s}\sim (\log N)^{-2}N^{-s}\gg 1$, and norm inflation occurs.

\underline{Case 6}: $Z=\T$, $(p,q)=(4,1)$ or $(4,2)$ or $(4,3)$, and $-\frac{1}{6}\le s<0$.

Take $\phi$ as in Lemma~\ref{lem:U_p} (iv), and then take $r=(\log N)^{-1}$ and $T=N^{3s}$, which implies
\eqq{T\ll 1,\quad \rho \sim rN^{-s}T^{\frac{1}{3}}=(\log N)^{-1}\ll 1,\quad r\rho N^{-s}\sim (\log N)^{-2}N^{-s}\gg 1\gg r.}
Again, we have $\tnorm{u(T)}{H^s}\sim \norm{U_4[\phi ](T)}{H^s}\sim (\log N)^{-4}N^{-s}\gg 1$.

\underline{Case 7}: $Z=\R ^d$ with $1\le d\le 3$, $(p,q)=(2,1)$, and $\frac{d}{2}-2\le s<-\frac{1}{4}$.

In this case the data $\phi$ is taken as in Lemma~\ref{lem:U_p} (ii) and does not satisfy \eqref{cond:phi}, so we need to modify the previous argument.

We use anisotropic modulation space $\ti{M}$ defined by the norm
\eqq{\norm{f}{\ti{M}}:=\sum _{\xi \in \Bo{Z}^{d-1}\times N^{-1}\Bo{Z}}\norm{\hhat{f}}{L^2(\xi +\ti{Q}_{N^{-1}})}.}
We have the product estimate 
\eqq{\tnorm{fg}{\ti{M}}\lec N^{-\frac{1}{2}}\tnorm{f}{\ti{M}}\tnorm{g}{\ti{M}}}
in this space.
Thus, we follow the proof of Lemma~\ref{lem:U_k} to obtain 
\eqq{\tnorm{U_k[\phi ](t)}{\ti{M}}\le Cr(CrN^{-\frac{1}{2}-s}t)^{k-1}N^{-s}}
for any $k\ge 1$, which is used to justify the expansion of the solution in $\ti{M}$ up to time $T$ such that $\ti{\rho}:=rN^{-\frac{1}{2}-s}T\ll 1$.
Then, by the same argument as in the proofs of Lemmas~\ref{lem:supp} and \ref{lem:U_k_H^s}, we see that 
\eqq{|\supp{\hhat{U_k[\phi ]}(t)}|\le C^kN^{-1},\qquad \tnorm{U_k[\phi ](T)}{H^s}\le Cr(C\ti{\rho})^{k-1}N^{-s}.}
In particular, $\tnorm{U_2[\phi ](T)}{H^s}\sim r\ti{\rho}N^{-s}$ for $0<T\ll 1$ by Lemma~\ref{lem:U_p} (iii).

Now, we take $r=(\log N)^{-1}\ll 1$, $T=(\log N)^3N^{2s+\frac{1}{2}}\ll 1$, so that $\ti{\rho}=(\log N)^2N^s\ll 1$, $r\ti{\rho}N^{-s}=\log N \gg r$.
From the estimates above, we have $\tnorm{u(T)}{H^s}\sim \log N \gg 1$, which shows norm inflation.
\end{proof}

%%%%%%%%%%%%%%%%%%%%%%%%%%%%%

\bigskip
% \newpage
\section{Proof of Theorem~\ref{thm:main}}\label{sec:proof}

Here, we see how to use the estimates for single-term nonlinearities for the proof in the multi-term cases.
We write $p:=\max _{1\le j\le n}p_j$.

For the initial value problem \eqref{NLS'}, the $k$-th order term $U_k[\phi ]$ in the expansion of the solution is given by $U_1[\phi ]:=e^{it\Delta}\phi$ and 
\eqq{U_k[\phi ]:=-i\sum _{j=1}^n\nu _j\sum _{\mat{k_1,\dots ,k_{p_j}\ge 1\\ k_1+\dots +k_{p_j}=k}}\int _0^t e^{i(t-\tau )\Delta}\mu _{p_j,q_j}\big( U_{k_1}[\phi ](\tau ),\dots ,U_{k_{p_j}}[\phi ](\tau )\big) \,d\tau }
for $k\ge 2$ inductively.

The following lemmas are verified in the same manner as Lemmas~\ref{lem:a_k}, \ref{lem:U_k}, and Corollary~\ref{cor:lwp}.
\begin{lem}\label{lem:a_k'}
Let $\shugo{b_k}_{k=1}^\I$ be a sequence of nonnegative real numbers such that 
\[ b_{k} \le \sum _{j=1}^nC_j\sum _{\mat{k_1,\dots ,k_{p_j}\ge 1\\ k_1+\dots +k_{p_j}=k}}b_{k_1}\cdots b_{k_{p_j}},\qquad k\ge 2\]
for some $p_1,\dots ,p_n\ge 2$ and $C_1,\dots ,C_n>0$.
Then, we have
\[ b_k\le b_1C_0^{k-1},\qquad k\ge 1,\qquad C_0=\max _{1\le j\le n}\frac{\pi ^2}{6}(nC_jp_j^2)^{\frac{1}{p_j-1}}b_1.\]
\end{lem}

\begin{lem}\label{lem:U_k'}
There exists $C>0$ such that for any $\phi \in M_A$ with $\tnorm{\phi}{M_A}\le M$ we have
\eqq{\norm{U_k[\phi ](t)}{M_A}\le t^{\frac{k-1}{p-1}}(CA^{\frac{d}{2}}M)^{k-1}M}
for any $0\le t\le 1$ and $k\ge 1$.
\end{lem}

\begin{lem}\label{lem:MAlwp'}
Let $\phi \in M_A$ with $\tnorm{\phi}{M_A}\le M$.
If $T>0$ satisfies $A^{\frac{d}{2}}MT^{\frac{1}{p-1}}\ll 1$, then a unique solution $u\in C([0,T];M_A)$ to \eqref{NLS'} exists and has the expansion $u=\sum _{k=1}^\I U_k[\phi ]$.
\end{lem}

The next lemma can be verified similarly to Lemma~\ref{lem:U_k_H^s}.
\begin{lem}\label{lem:U_k_H^s'}
Let $\phi$ satisfy \eqref{cond:phi} and $s<0$.
Then, the following holds.
\begin{enumerate}
\item $\norm{U_1[\phi ](T)}{H^s}\le Cr$\hx for any $T\ge 0$.
\item $\norm{U_k[\phi ](T)}{H^s}\le Cr(C\rho)^{k-1}A^{-\frac{d}{2}}N^{-s}f_s(A)$\hx for any $0\le T\le 1$ and $k\ge 2$, where
\eqq{\rho =rA^{\frac{d}{2}}N^{-s}T^{\frac{1}{p-1}}\quad (p=\max _{1\le j\le n}p_j),\qquad f_s(A)=\norm{\LR{\xi}^s}{L^2(\shugo{|\xi |\le A})}.}
\end{enumerate}
\end{lem}

We now begin to prove Theorem~\ref{thm:main}.

\begin{proof}[Proof of Theorem~\ref{thm:main}]
We divide the proof into two cases:
(I) One of the terms of order $p$ (highest order) is responsible for norm inflation, or (II) a lower order term determines the range of regularities for norm inflation. 
Note that (II) occurs only when $Z=\R$, $p=3$, $F (u,\bar{u})$ has the term $u\bar{u}$ and $s\in (-\frac{1}{2},-\frac{1}{4})$.

(I): Rewrite the nonlinear terms as
\eqq{F (u,\bar{u})=\sum _{q=0}^p\nu _{p,q}\mu _{p,q}(u)+\text{(terms of order less than $p$)}.}
Note that $\nu _{p,q}$ may be zero but $(\nu _{p,0},\dots ,\nu _{p,p})\neq (0,\dots ,0)$.

We divide the series into four parts:
\eqq{\sum _{k=1}^\I U_k[\phi ]&=U_1[\phi ]+\Big\{ \sum _{k=2}^{p}U_k[\phi ]-\Big( -i\sum _{q=0}^p\nu _{p,q}\int _0^te^{i(t-\tau )\Delta}\mu _{p,q}\big( U_1[\phi ](\tau )\big) \,d\tau \Big) \Big\} \\
&\hx +\Big( -i\sum _{q=0}^p\nu _{p,q}\int _0^te^{i(t-\tau )\Delta}\mu _{p,q}\big( U_1[\phi ](\tau )\big) \,d\tau \Big) +\sum _{k=p+1}^\I U_k[\phi ]\\
&=:U_1[\phi ]+U_{low}[\phi ]+U_{main}[\phi ]+U_{high}[\phi ].}
Note that $U_{low}=0$ if $p=2$.

The following lemma indicates how $U_{low}$ is dominated by $U_{main}$, and how the contributions of the $(p+1)$ terms in $U_{main}$ can be `separated'.
\begin{lem}\label{lem:U_p'}

We have the following:
\begin{enumerate}
\item Let $\phi$ satisfy \eqref{cond:phi} and $s<0$.
Let $0<T\le 1$, and assume that $\rho =rA^{\frac{d}{2}}N^{-s}T^{\frac{1}{p-1}}\ll 1$.
Then, (if $p\ge 3$,) 
\eqq{\norm{U_{low}[\phi](T)}{H^s}\lec r^2N^{-2s}f_s(A)T^{\frac{1}{p-2}}.}
\item Let $q_*\in \shugo{0,1,\dots ,p}$ be such that $\nu _{p,q_*}\neq 0$.
Then, for any $T\ge 0$ there exists $j\in \shugo{0,1,\dots ,p}$ such that
\eqq{\norm{U_{main}[e^{i\frac{j\pi}{p+1}}\phi ](T)}{H^s}\gec \tnorm{G_{q_*}[\phi ](T)}{H^s},}
where
\eqq{G_q[\phi ](t):=-i\int _0^te^{i(t-\tau )\Delta}\mu _{p,q}(U_1[\phi ](\tau ))\,d\tau ;%\quad \text{so that}
\quad U_{main}[\phi ]=\sum _{q=0}^p\nu _{p,q}G_q[\phi ](t).}
\end{enumerate}
\end{lem}

\begin{proof}
(i) We notice that the nonlinear terms of highest order $p$ have nothing to do with $U_{low}[\phi ]$. 
Hence, we estimate by Lemma~\ref{lem:U_k_H^s'} (ii) with $p$ replaced by $p-1$ and have
\eqq{\norm{U_{low}[\phi ](T)}{H^s}\le \sum _{k=2}^pCr(CrA^{\frac{d}{2}}N^{-s}T^{\frac{1}{(p-1)-1}})^{k-1}A^{-\frac{d}{2}}N^{-s}f_s(A).}
Since $0<T\le 1$ implies $rA^{\frac{d}{2}}N^{-s}T^{\frac{1}{(p-1)-1}}\le \rho \ll 1$, we have
\eqq{\norm{U_{low}[\phi ](T)}{H^s}\lec r\cdot rA^{\frac{d}{2}}N^{-s}T^{\frac{1}{p-2}}\cdot A^{-\frac{d}{2}}N^{-s}f_s(A).}

(ii)
We observe that $\zeta _p:=e^{i\frac{\pi}{p+1}}$ satisfies $\sum _{j=0}^{p}\zeta _p^{2qj}=0$ if $q\not\equiv 0\mod p+1$.
Since $G_q[\zeta _p^j\phi ]=\zeta _p^{(p-2q)j}G_q[\phi ]$, for any $0\le q_*\le p$ it holds that
\eqq{\sum _{j=0}^p\zeta _p^{(2q_*-p)j}U_{main}[\zeta _p^j\phi ]&=\sum _{q=0}^{p}\sum _{j=0}^p\zeta _p^{2(q_*-q)j}\nu _{p,q}G_q[\phi ]=(p+1)\nu _{p,q_*}G_{q_*}[\phi ].}
Hence, if $\nu _{p,q_*}\neq 0$, by the triangle inequality we see that 
\eqq{\sum _{j=0}^p\norm{U_{main}[\zeta _p^j\phi ](T)}{H^s}\ge (p+1)|\nu _{p,q_*}|\norm{G_{q_*}[\phi ](T)}{H^s}.
}
This implies the claim.
\end{proof}

By Lemma~\ref{lem:U_p'}, the proof is almost reduced to the case of single-term nonlinearities, as we see below.

\underline{Case 1}: General $Z$ and $p$, $s<\min \shugo{s_c(d,p),0}$.

Let us take the initial data $\phi$ as in Lemma~\ref{lem:U_p} (i), and assume $\rho =rA^{-\frac{d}{2}}N^{-s}T^{\frac{1}{p-1}}\ll 1$, $0<T\ll N^{-2}$.
Lemma~\ref{lem:U_k_H^s'} (ii) yields that
\eqq{\norm{U_{high}[\zeta _p^j\phi ](T)}{H^s}\lec r\rho ^pA^{-\frac{d}{2}}N^{-s}f_s(A),}
while Lemma~\ref{lem:U_p'} (ii) and Lemma~\ref{lem:U_p} (i) imply that
\eqq{\norm{U_{main}[\zeta _p^j\phi ](T)}{H^s}\sim r\rho ^{p-1}A^{-\frac{d}{2}}N^{-s}f_s(A)\gg \norm{U_{high}[\zeta _p^j\phi ](T)}{H^s}}
for an appropriate $j$.
Hence, from Lemma~\ref{lem:U_k_H^s'} (i) and Lemma~\ref{lem:U_p'} (i),
\eqq{\norm{u(T)}{H^s}&\ge \tfrac{1}{2}\norm{U_{main}[\zeta _p^j\phi ](T)}{H^s}-\norm{U_{low}[\zeta _p^j\phi ](T)}{H^s}-\norm{U_1[\zeta _p^j\phi ](T)}{H^s}\\
&\ge C^{-1}r\rho ^{p-1}A^{-\frac{d}{2}}N^{-s}f_s(A)-C\big( r^2N^{-2s}f_s(A)T^{\frac{1}{p-2}}+r\big) .}
If we take the same choice for $r,A,T$ as in Case~1 of the proof of Theorem~\ref{thm:main0}; 
\eqq{r=(\log N)^{-1},\quad  A\sim (\log N)^{-\frac{p+1}{|s|}}N,\quad T=(A^{-\frac{d}{2}}N^s)^{p-1};\quad \text{so that}\hx \rho = (\log N)^{-1},}
all the required conditions for norm inflation are satisfied when $p=2$.
Even for $p\ge 3$, it suffices to check that
\eqq{r\rho ^{p-1}A^{-\frac{d}{2}}N^{-s}f_s(A)\gg r^2N^{-2s}f_s(A)T^{\frac{1}{p-2}}.}
This is equivalent to $\rho ^{p-2}\gg T^{\frac{1}{p-2}-\frac{1}{p-1}}$, which we can easily show.

\underline{Case 2-4-5-7}: $p=2$.
We need to deal with the following situations:
\begin{itemize}
\item $d=1$, $\nu _{2,1}=0$, $-\frac{3}{2}\le s<-1$;
\item $d=2$, $\nu _{2,1}=0$, $s=-1$;
\item $Z=\R ^{d_1}\times \T ^{d_2}$ with $d_1+d_2\le 3$, $d_2\ge 1$, $\nu _{2,1}\neq 0$, and $\frac{d}{2}-2\le s<0$;
\item $Z=\R ^d$, $1\le d\le 3$, $\nu _{2,1}\neq 0$, $\frac{d}{2}-2\le s<-\frac{1}{4}$,
\end{itemize}
which correspond to Cases 2, 4, 5, and 7 in the proof of Theorem~\ref{thm:main0}, respectively.
As seen in the preceding case, we do not have to care about $U_{low}$ and the proof is the same as the single-term cases, except that we need to pick up the appropriate one among $u^2$, $u\bar{u}$, $\bar{u}^2$ by using Lemma~\ref{lem:U_p'} (ii).

\underline{Case 3}: $d=1$, $p=3$, $s=-\frac{1}{2}$.

We take the initial data $e^{i\frac{j\pi}{4}}\phi$ with $\phi$ as in \eqref{cond:phi} and parameters $r,A,T$ as in Case~3 for Theorem~\ref{thm:main0}.
Following the argument in Case 1, it suffices to check the condition for $\tnorm{U_{main}}{H^s}\gg \tnorm{U_{low}}{H^s}$;
\eqq{r\rho ^2A^{-\frac{1}{2}}N^{\frac{1}{2}}f_{-\frac{1}{2}}(A)\gg r^2Nf_{-\frac{1}{2}}(A)T.}
Actually, we see that $\text{L.H.S.}\sim (\log N)^{\frac{1}{24}}\gg (\log N)^{\frac{1}{4}}N^{-1}\sim \text{R.H.S.}$

\underline{Case 6}: $Z=\T$, $p=4$, $(\nu _{4,1},\nu _{4,2},\nu _{4,3})\neq (0,0,0)$, $s\in [-\frac{1}{6},0)$.

Similarly, we take $e^{i\frac{j\pi}{5}}\phi$ with parameters $r,A,T$ as in Case~6 for Theorem~\ref{thm:main0}.
It suffices to verify the condition
\eqq{r\rho ^{3}N^{-s}\gg r^2N^{-2s}T^{\frac{1}{2}},}
and in fact it holds that $\text{L.H.S.}\sim (\log N)^{-4}N^{-s}\gg (\log N)^{-2}N^{-\frac{s}{2}}\sim \text{R.H.S.}$

(II): Recall that we claim NI$_s$ for $s\in (-\frac{1}{2},-\frac{1}{4})$ in the case of $Z=\R$, $p=3$, and $F (u,\bar{u})$ has the term $u\bar{u}$. 

We take $\phi$ as in \eqref{cond:phi} with $A=N^{-1}$ and $\Sigma =\shugo{N}$ (same as in Case 7 for the single-term nonlinearity).
By Lemmas~\ref{lem:MAlwp'} and \ref{lem:U_k_H^s'}, we can expand the solution whenever $\rho =rN^{-\frac{1}{2}-s}T^{\frac{1}{2}}\ll 1$ and we have 
\eqq{\sum _{k\ge 4}\tnorm{U_k[\phi ](T)}{H^s}\lec r\rho ^3N^{\frac{1}{2}-s}f_s(N^{-1})\sim r^4N^{-\frac{3}{2}-4s}T^{\frac{3}{2}}}
for $0<T\le 1$.
For $U_3$, observing that the Fourier support is in the region $|\xi |\sim N$, we modify the estimate in Lemma~\ref{lem:U_k_H^s'} to obtain 
\eqq{\tnorm{U_3[\phi ](T)}{H^s}\lec r\rho ^2N^{\frac{1}{2}-s}\cdot \tnorm{\LR{\xi}^s}{L^2(\supp \hhat{U_3[\phi ]})}\sim r^3N^{-1-2s}T.}
For $U_2$ the contribution from $u^2$ and $\bar{u}^2$ has the Fourier support in high frequency, thus being dominated by the contribution from $u\bar{u}$.
By Lemma~\ref{lem:U_p} (ii), we have 
\eqq{\tnorm{U_2[\phi ](T)}{H^s}\gec r^2N^{-\frac{1}{2}-2s}T}
if $0<T\ll 1$.
We set $r=(\log N)^{-1}$ and $T=(\log N)^3N^{2s+\frac{1}{2}}$ as before (Case 7 in the single-term case), then it holds that $T\ll 1$, $\rho =(\log N)^{\frac{1}{2}}N^{-\frac{1}{4}}\ll 1$ and 
\eqq{\tnorm{u(T)}{H^s}\ge C^{-1}r^2N^{-\frac{1}{2}-2s}T-C\big( r+r^3N^{-1-2s}T+r^4N^{-\frac{3}{2}-4s}T^{\frac{3}{2}}\big) \gec \log N\gg 1}
for $s\in [-\frac{3}{4},-\frac{1}{4})$, which gives the claimed norm inflation.

This concludes the proof of Theorem~\ref{thm:main}.
\end{proof}

%%%%%%%%%%%%%%%%%%%%%%%%%%%%%

\bigskip
% \newpage
\appendix
\section{Norm inflation with infinite loss of regularity}\label{sec:niilr}

In this section, we derive norm inflation with infinite loss of regularity for the problem with smooth gauge-invariant nonlinearities:
\begin{equation}\label{nuNLS}
\left\{
\begin{array}{@{\,}r@{\;}l}
i\p _tu+\Delta u&=\pm |u|^{2\nu}u,\qquad t\in [0,T],\quad x\in Z=\R ^{d-d_2}\times \T^{d_2},\\
u(0,x)&=\phi (x),
\end{array}
\right.
\end{equation}
where $\nu$ is a positive integer.
The initial value problem \eqref{nuNLS} on $\R ^d$ is invariant under the scaling $u(t,x)\mapsto \la ^{\frac{1}{\nu}}u(\la ^2t,\la x)$, and the critical Sobolev index is $s_c(d,2\nu +1)=\frac{d}{2}-\frac{1}{\nu}$, which is non-negative except for the case $d=\nu =1$.

\begin{prop}\label{prop:niilr}
We assume the following condition on $s$:
\begin{itemize}
\item If $d=\nu =1$, then $s<-\frac{2}{3}$;
\item if $d\ge 2$, $\nu =1$ and $d_2=0,1$ (i.e., $Z=\R ^d$ or $\R ^{d-1}\times \T$), then $s<-\frac{1}{3}$;
\item if $d\ge 1$, $\nu \ge 2$ and $d_2=0$ (i.e., $Z=\R ^d$), then $s<-\frac{1}{2\nu +1}$;
\item otherwise, $s<0$.
\end{itemize}
Then, NI$_s$ with infinite loss of regularity occurs for the initial value problem \eqref{nuNLS}: 
For any $\de >0$ there exist $\phi \in H^\I$ and $T>0$ satisfying $\tnorm{\phi}{H^s}<\de$, $0<T<\de$ such that the corresponding smooth solution $u$ to \eqref{NLS} exists on $[0,T]$ and $\tnorm{u(T)}{H^\sgm}>\de ^{-1}$ for \emph{all} $\sgm \in \R$.%
\footnote{%
More precisely, we show $\tnorm{\hhat{u}(T)}{L^2(\{ |\xi |\le 1\} )}>\de ^{-1}$.
This implies the claim if we define the Sobolev norm of negative indices $\sigma$ as $\tnorm{f}{H^\sgm}:=\tnorm{\min \{ 1,\,|\xi |^{\sgm}\} \hhat{f}(\xi )}{L^2}$.
}
\end{prop}

\begin{rem}
(i) The proofs of Theorems~\ref{thm:main0} and \ref{thm:main} are easily adapted to yield NI$_s$ with \emph{finite} loss of regularity in most cases.
However, we only consider here \emph{infinite} loss of regularity.

(ii) The coefficient of the nonlinearity is not important in the proof, and the same result holds for any non-zero complex constant.

(iii) To show infinite loss of regularity, we need to use the nonlinear interactions of very high frequencies which create a significant output in low frequency $\{ |\xi |\le 1\}$.
Except for the case $d=\nu =1$, there are such interactions that are also \emph{resonant}; i.e., there exist non-zero vectors $k_1,\dots ,k_{2\nu +1}\in \Bo{Z}^{d}$ satisfying
\eqq{\sum _{j=0}^{\nu}k_{2j+1}=\sum _{l=1}^{\nu}k_{2l},\qquad \sum _{j=0}^{\nu}|k_{2j+1}|^2=\sum _{l=1}^{\nu}|k_{2l}|^2.}
This is also the key ingredient in the proof of the previous results \cite{CDS12,CK17}, and hence the restriction on the range of $s$ in Proposition~\ref{prop:niilr} is the same as that in \cite{CDS12,CK17}.

A complete characterization of the resonant set
\eqq{\Sc{R}_{d,\nu}(k):=\Shugo{(k_m)_{m=1}^{2\nu +1}\in (\Bo{Z}^d)^{2\nu +1}}{k=\sum _{m=1}^{2\nu +1}(-1)^{m+1}k_m,\, |k|^2=\sum _{m=1}^{2\nu +1}(-1)^{m+1}|k_m|^2}}
(for $k\in \Bo{Z}^d$ given) is easily obtained in the $\nu =1$ case; see \cite[Proposition~4.1]{CK17} for instance.
In Proposition~\ref{prop:char} below, we will provide a complete characterization of the set $\Sc{R}_{1,2}(0)$, which may be of interest in itself.
Since $(k_m)_{m=1}^5\in \Sc{R}_{1,2}(k)$ if and only if $(k_m-k)_{m=1}^5\in \Sc{R}_{1,2}(0)$, we have a characterization of $\Sc{R}_{1,2}(k)$ for any $k\in \Bo{Z}$ as well.
However, in the proof of Proposition~\ref{prop:niilr} we only need the fact that $\Sc{R}_{d,\nu}(0)$ has an element consisting of non-zero vectors in $\Bo{Z}^d$, except for $(d,\nu )=(1,1)$.
\end{rem}

\begin{proof}[Proof of Proposition~\ref{prop:niilr}]
We follow the proof of Theorem~\ref{thm:main0} but take different initial data to show infinite loss of regularity.
% It suffices to consider the case $\sgm \le s$.

Let $N\gg 1$ be a large positive integer and define $\phi \in H^\infty (Z)$ by
\eqq{\hhat{\phi}:=rN^{-s}\chi _{\Sigma +Q_1},}
where $r=r(N)>0$ is a constant to be chosen later, $Q_1:=[-\tfrac{1}{2},\tfrac{1}{2})^d$, and
\eqs{\Sigma :=
\begin{cases}
\shugo{N,2N} &\text{if $d=\nu =1$},\\
\shugo{Ne_{d-1},\,Ne_d,\,N(e_{d-1}+e_d)} &\text{if $d\ge 2$, $\nu =1$},\\
\shugo{Ne_d,\,3Ne_d,\,4Ne_d} &\text{if $d\ge 1$, $\nu \ge 2$},
\end{cases}\\
e_d:=(\underbrace{0,\dots ,0}_{d-1},1),\qquad e_{d-1}:=(\underbrace{0,\dots ,0}_{d-2},1,0).
}
The argument in Section~\ref{sec:proof0} (with $A=1$) shows the following:
\begin{itemize}
\item The unique solution $u=u[\phi ]$ to \eqref{nuNLS} exists on $[0,T]$ and has the power series expansion $u=\sum _{k=1}^\I U_k[\phi ]$ if $\rho :=rN^{-s}T^{\frac{1}{2\nu}}\ll 1$.
\item $\tnorm{U_1[\phi ](T)}{H^s}=\tnorm{\phi}{H^s}\sim r$ for any $T\ge 0$.
\item $\tnorm{U_k[\phi ](T)}{H^s}\le C\rho ^{k-1}rN^{-s}$ for any $T\ge 0$ and $k\ge 2$.
\end{itemize}

For the first nonlinear term $U_{2\nu +1}[\phi]$, we observe that
\eqq{|\hhat{U_{2\nu +1}[\phi ]}(T,\xi )|&=c(rN^{-s})^{2\nu +1}\Big| \int _\Gamma \prod _{m=1}^{2\nu+1}\chi _{\Sigma +Q_1}(\xi _m) \Big( \int _0^Te^{it\Phi}\,dt\Big) \,d\xi _1\dots d\xi _{2\nu +1}\Big| ,}
where
\eqq{\Gamma :=\Shugo{(\xi _1,\dots ,\xi _{2\nu +1})}{\sum _{j=0}^\nu \xi _{2j+1}-\sum _{l=1}^\nu \xi _{2l}=\xi},\quad \Phi :=|\xi |^2-\sum _{j=0}^\nu |\xi _{2j+1}|^2+\sum _{l=1}^\nu |\xi _{2l}|^2.}
Now, we restrict $\xi$ to the low-frequency region $Q_{1/2}$.
If $d=\nu =1$, then we have
\eqq{&\chi _{Q_{1/2}}(\xi )\int _\Gamma \prod _{m=1}^{2\nu+1}\chi _{\Sigma +Q_1}(\xi _m) \int _0^Te^{it\Phi}\,dt\\
&=2\chi _{Q_{1/2}}(\xi )\int _\Gamma \chi _{N+Q_1}(\xi _1)\chi _{2N+Q_1}(\xi _2)\chi _{N+Q_1}(\xi _3)\int _0^Te^{it\Phi}\,dt,}
and $\Phi =O(N^2)$ in the integral.
If $d\ge 2$ and $\nu =1$, we have
\eqq{&\chi _{Q_{1/2}}(\xi )\int _\Gamma \prod _{m=1}^{2\nu+1}\chi _{\Sigma +Q_1}(\xi _m) \int _0^Te^{it\Phi}\,dt\\
&=2\chi _{Q_{1/2}}(\xi )\int _\Gamma \chi _{Ne_{d-1}+Q_1}(\xi _1)\chi _{N(e_{d-1}+e_d)+Q_1}(\xi _2)\chi _{Ne_d+Q_1}(\xi _3)\int _0^Te^{it\Phi}\,dt,}
and the resonant property implies that
\eqq{
\Phi =\begin{cases}
O(N) &\text{if $d_2=0,1$},\\
O(1) &\text{if $d_2\ge 2$}
\end{cases}
}
in the integral.
Therefore, in these cases we have the following lower bound:
\eq{est:A}{\norm{\hhat{U_{2\nu +1}[\phi ]}(T)}{L^2(Q_{1/2})}\ge cT(rN^{-s})^{2\nu +1}=c\rho ^{2\nu}rN^{-s}}
\eqq{\text{for any}\quad
0<T\ll \begin{cases}
N^{-2} &\text{if $d=\nu =1$},\\
N^{-1} &\text{if $d\ge 2$, $\nu =1$, $d_2=0,1$},\\
~1 &\text{if $d\ge 2$, $\nu =1$, $d_2\ge 2$}.
\end{cases}
}
The quintic and higher cases are slightly different.
On one hand, there are ``almost resonant'' interactions such as
\eqq{\prod _{j=1,3}\chi _{Ne_d+Q_1}(\xi _j)\prod _{l=2,4}\chi _{3Ne_d+Q_1}(\xi _l)\prod _{m=5}^{2\nu +1}\chi _{4Ne_d+Q_1}(\xi _m),}
for which it holds 
\eqq{
\Phi =\begin{cases}
O(N) &\text{if $d_2=0$},\\
O(1) &\text{if $d_2\ge 1$}
\end{cases}
}
in the integral.
On the other hand, some non-resonant interactions such as
\eqq{\chi _{3Ne_d+Q_1}(\xi _1)\chi _{4Ne_d+Q_1}(\xi _2)\prod _{m=3}^{2\nu +1}\chi _{Ne_d+Q_1}(\xi _m)}
also create low-frequency modes, with $|\Phi |\sim N^2$ in the integral.
Hence, if we choose $T>0$ as
\eqq{
N^{-2}\ll T\ll \begin{cases}
N^{-1} &\text{if $d\ge 1$, $\nu \ge 2$, $d_2=0$},\\
~1 &\text{if $d\ge 1$, $\nu \ge 2$, $d_2\ge 1$},
\end{cases}
}
then 
\eqq{
\begin{cases}
\Re \Big( \displaystyle\int _0^Te^{it\Phi}\,dt\Big) \ge \frac{1}{2}T &\text{for ``almost resonant'' interactions},\\[10pt]
\Big| \displaystyle\int _0^Te^{it\Phi}\,dt\Big| \le CN^{-2}\ll T &\text{for non-resonant interactions},
\end{cases}
}
so that no cancellation occurs among ``almost resonant'' interactions, which dominate the non-resonant interactions.
Therefore, we have \eqref{est:A} for such $T$ as above.

Finally, we set
\eqq{
\begin{cases}
r:=N^{s+\frac{2}{3}}\log N,\quad T:=N^{-2}(\log N)^{-1} &\text{if $d=\nu =1$},\\
r:=N^{s+\frac{1}{2\nu +1}}\log N,\quad T:=N^{-1}(\log N)^{-1} &\text{if $d\ge 2$, $\nu =1$, $d_2=0,1$}\\[-5pt]
&\quad \text{or $d\ge 1$, $\nu \ge 2$, $d_2=0$},\\
r:=N^s\log N,\quad T:=(\log N)^{-(2\nu +\frac{1}{2})} &\text{otherwise}.
\end{cases}
}
We see that, under the assumption on $s$, $\tnorm{\phi}{H^s}\sim r\ll 1$, $T\ll 1$, $\rho \ll 1$, and
\eqq{\norm{\hhat{u}(T)}{L^2(Q_{1/2})}&\ge c\norm{\hhat{U_{2\nu +1}[\phi ]}(T)}{L^2(Q_{1/2})}-C\Big( \norm{U_1[\phi ](T)}{H^s}+\sum _{l\ge 2} \norm{U_{2\nu l+1}[\phi ](T)}{H^s}\Big) \\
&\ge c\norm{\hhat{U_{2\nu +1}[\phi ]}(T)}{L^2(Q_{1/2})}\gg 1.}
We conclude the proof by letting $N\to \I$.
\end{proof}

At the end of this section, we give a characterization of resonant interactions creating the zero mode in the one-dimensional quintic case.
\begin{prop}\label{prop:char}
The quintuplet $(k_1,\dots ,k_5)\in \Bo{Z}^5$ satisfies 
\eq{cond:res}{k_1+k_3+k_5=k_2+k_4,\qquad k_1^2+k_3^2+k_5^2=k_2^2+k_4^2}
if and only if 
\eq{char}{\shugo{k_1,\,k_3,\,k_5}&=\shugo{ap,\,bq,\,(a+b)(p+q)},\\
\shugo{k_2,\,k_4}&=\shugo{ap+(a+b)q,\,(a+b)p+bq}}
for some $a,b,p,q\in \Bo{Z}$.
\end{prop}
\begin{exm}
(i) Taking $a=p=b=q=1$ in \eqref{char}, we have the quintuplet %$(1,3,1,3,4)$. Also, with $a=p=-b=1$ and $q=-2$ we have 
$(1,3,1,3,4)$ which has appeared in the proof of Proposition~\ref{prop:niilr} above.
Also, with $(a,b,p,q)=(-1,2,-2,1)$ we have $(2,3,2,0,-1)$, which gives a resonant interaction for quartic nonlinearities $u^3\bar{u}$, $u\bar{u}^3$ exploited in the proof of Lemma~\ref{lem:U_p}~(iv) above.

(ii) The quintuplets $(pq,-q^2,-pq,p^2,p^2-q^2)$ given in \cite[Lemma~4.2]{CK17} can be obtained by setting $a=-q$, $b=p$ in \eqref{char}.
\end{exm}
\begin{proof}[Proof of Proposition~\ref{prop:char}]
The \emph{if} part is verified by a direct computation, so we show the \emph{only if} part.

Let $(k_1,\dots ,k_5)\in \Bo{Z}^5$ satisfy \eqref{cond:res}.
We start with observing that at least one of $k_1,k_3,k_5$ is an even integer; otherwise, we would have
\eqq{k_1^2+k_3^2+k_5^2\equiv 3\not\equiv 1\equiv k_2^2+k_4^2\mod 4,}
contradicting \eqref{cond:res}.
Without loss of generality, we assume $k_5$ to be even and set
\eqq{n_j:=k_j-\tfrac{1}{2}k_5\in \Bo{Z}\quad (j=1,\dots ,5),\qquad n_6:=-\tfrac{1}{2}k_5\in \Bo{Z}.}
From \eqref{cond:res} we see that
\eqq{n_1+n_3+n_5=n_2+n_4+n_6,\quad n_1^2+n_3^2=n_2^2+n_4^2,\quad n_5=-n_6.}
The second equality implies that two vectors $(n_1-n_2,n_3-n_4), (n_1+n_2,n_3+n_4)\in \Bo{Z}^2$ are orthogonal to each other (unless one of them is zero), which allows us to write
\eq{id:A}{(n_1-n_2,n_3-n_4)=\al (q,p),\quad (n_1+n_2,n_3+n_4)=\be (-p,q)}
with $\al ,\be ,p,q\in \Bo{Z}$.
Note that $n_1,\dots ,n_4$ are then written as
\eqq{n_1=\tfrac{1}{2}(\al q-\be p),\qquad n_2=-\tfrac{1}{2}(\al q+\be p),\\
n_3=\tfrac{1}{2}(\al p+\be q),\qquad n_4=-\tfrac{1}{2}(\al p-\be q),}
and that
\eqq{n_5=-n_6=\tfrac{1}{2}(n_5-n_6)=-\tfrac{1}{2}\big\{ (n_1-n_2)+(n_3-n_4)\big\} =-\tfrac{1}{2}\al (p+q).}
Recalling $k_j=n_j-n_6$ ($j=1,\dots ,5$), we have
\eq{ks}{&k_1=-\tfrac{1}{2}(\al +\be )p,\qquad k_3=-\tfrac{1}{2}(\al -\be )q,\qquad k_5=-\al (p+q),\\
&\qquad k_2=-\tfrac{1}{2}(\al +\be )p-\al q,\qquad k_4=-\tfrac{1}{2}(\al -\be )q-\al p.}

We next claim that the integers $\al ,\be ,p,q$ can be chosen in \eqref{id:A} so that $\al$ and $\be$ have the same parity.
To see this, we notice that the four integers $n_1\pm n_2$, $n_3\pm n_4$ are of the same parity, since all of
\eqs{(n_1+n_2)+(n_1-n_2)=2n_1,\qquad (n_3+n_4)+(n_3-n_4)=2n_3,\\
(n_1-n_2)+(n_3-n_4)=n_6-n_5=2n_6}
are even.
If $n_1\pm n_2$, $n_3\pm n_4$ are odd integers, then by \eqref{id:A} $\al$ and $\be$ must be odd.
So, we assume that they are all even.
If one of $p,q$ is odd, then both $\al$ and $\be$ must be even.
If both $p$ and $q$ are even, we replace $(\al ,\be ,p,q)$ with $(2\al ,2\be ,p/2,q/2)$ to obtain another expression \eqref{id:A} with both $\al$ and $\be$ being even.
Hence, the claim is proved.

Finally, we set $a:=-\frac{1}{2}(\al +\be )$, $b:=-\frac{1}{2}(\al -\be )$, both of which are integers.
Inserting them into \eqref{ks}, we find the expression \eqref{char}.
\end{proof}

%%%%%%%%%%%%%%%%%%%%%%%%%%%%%

\bigskip
% \newpage
\section{Norm inflation for 1D cubic NLS at the critical regularity}\label{sec:ap}

In this section, we consider the particular equation
\begin{equation}\label{cNLS}
\left\{
\begin{array}{@{\,}r@{\;}l}
i\p _tu+\p _x^2 u&=\pm |u|^2u,\qquad t\in [0,T],\quad x\in Z=\R \text{~or~} \T ,\\
u(0,x)&=\phi (x).
\end{array}
\right.
\end{equation}
% The initial value problem \eqref{cNLS} is invariant under the scaling $u(t,x)\mapsto \la u(\la ^2t,\la x)$, and the critical Sobolev index is $-\frac{1}{2}$, so that $\tnorm{\la u(0,\la \cdot )}{\dot{H}^{-1/2}}=\tnorm{u(0,\cdot )}{\dot{H}^{-1/2}}$.
We will show the inflation of the Besov-type scale-critical Sobolev and Fourier-Lebesgue norms with an additional logarithmic factor:
\begin{defn}
For $1\le p<\I$, $1\le q\le \I$ and $\al \in \R$, define the $D^{[\al]}_{p,q}$-norm by
\eqq{\norm{f}{D^{[\al]}_{p,q}}:=\Big\| N^{-\frac{1}{p}}\LR{\log N}^\al \norm{\hhat{f}}{L^p_\xi (\{ N\le \LR{\xi}<2N\} )}\Big\| _{\ell ^q_N(2^{\Bo{Z}_{\ge 0}})} .}
We also define the $D^{s}_{p,q}$-norm for $s\in \R$ by
\eqq{\norm{f}{D^{s}_{p,q}}:=\Big\| N^{s}\norm{\hhat{f}}{L^p_\xi (\{ N\le \LR{\xi}<2N\} )}\Big\| _{\ell ^q_N(2^{\Bo{Z}_{\ge 0}})} .}
\end{defn}
\begin{rem}
(i) We see that $D^{[0]}_{2,q}=D^{-\frac{1}{2}}_{2,q}=B^{-\frac{1}{2}}_{2,q}$ (Besov norm) and $D^{[0]}_{p,p}=\F L^{-\frac{1}{p},p}$ (Fourier-Lebesgue norm).
In the case of $Z=\R$, the homogeneous version of $D^{[0]}_{p,q}$ is scale invariant for any $p,q$.

(ii) We have the embeddings $D^{[\al ]}_{p_2,q}\hookrightarrow D^{[\al]}_{p_1,q}$ if $p_1\le p_2$, 
$D^{[\al ]}_{p,q_1}\hookrightarrow D^{[\al ]}_{p,q_2}$ if $q_1\le q_2$.

(iii) We will not consider the space $D^{[\al ]}_{p,q}$ with $p=\I$ here, since our argument seems valid only in the space of negative regularity.
\end{rem}

\begin{prop}\label{prop:A}
For the Cauchy problem \eqref{cNLS}, norm inflation occurs in the following cases:

(i) In $D^{[\al ]}_{p,q}$ for any $1\le q\le \I$ and $\al <\frac{1}{2q}$, if $\frac{3}{2}\le p<\I$.

(ii) In $D^{[\al ]}_{p,q}$ and $D^s_{p,q}$ for any $1\le q\le \I$, $\al \in  \R$ and $s<-\frac{2}{3}$, if $1\le p<\frac{3}{2}$.
\end{prop}
\begin{rem}
(i) If $\frac{3}{2}\le p<\I$ and $1\le q<\I$, Proposition~\ref{prop:A} shows inflation of a ``logarithmically subcritical'' norm (i.e., $D^{[\al ]}_{p,q}$ with $\al >0$).
Moreover, if $1\le p<\frac{3}{2}$ we show norm inflation in $D^{s}_{p,q}$ for subcritical regularities $-\frac{2}{3}>s>-\frac{1}{p}$. 
However, for $q=\I$ and $p\ge \frac{3}{2}$, inflation is not detected even in the critical norm $D^{[0]}_{p,\I}$.

(ii) In \cite[Theorem~4.7]{KVZ17p} global-in-time a priori bound was established in $D^{[\frac{3}{2}]}_{2,2}$ and $D^{[2]}_{2,\I}$.
Recently, Oh and Wang \cite{OW18p} proved global-in-time bound in $\F L^{0,p}$ for $Z=\T$ and $2\le p<\I$.
There are still some gaps between these results and ours.
In fact, Proposition~\ref{prop:A} shows inflation of $D^{[\frac{1}{4}-]}_{2,2}$ and $D^{[0-]}_{2,\I}$ norms, as well as in a norm only logarithmically stronger than $\F L^{-\frac{1}{p},p}$ for $p\ge 2$.

(iii) Guo \cite{G17} also studied \eqref{cNLS} on $\R$ in ``almost critical'' spaces.
It would be interesting to compare our result with \cite[Theorem~1.8]{G17}, where he showed well-posedness (and hence a priori bound) in some Orlicz-type generalized modulation spaces which are barely smaller than the critical one $M_{2,\I}$.
There is no conflict between these results, because the function spaces for which norm inflation is claimed in Proposition~\ref{prop:A} are not included in $M_{2,\I}$ due to negative regularity.
Note also that the function spaces in \cite[Theorem~1.8]{G17} admit the initial data $\phi$ of the form $\hhat{\phi}(\xi )=[\log (2+|\xi |)]^{-\gamma}$ only for $\ga >2$ (see \cite[Remark~1.9]{G17}), while it belongs to $D_{p,q}^{[\al ]}$ if $\ga >\al +\frac{1}{q}$.

(iv) In contrast to the results in \cite{KVZ17p,OW18p}, complete integrability of the equation will play no role in our argument.
In particular, Proposition~\ref{prop:A} still holds if we replace the nonlinearity in \eqref{cNLS} with any of the other cubic terms $u^3,\bar{u}^3,u\bar{u}^2$ or any linear combination of them with complex coefficients.
\end{rem}

\begin{proof}[Proof of Proposition~\ref{prop:A}]
We follow the argument in Section~\ref{sec:proof0}.
For $1\le \rho <\I$ and $A>0$, let $M^\rho _A$ be the rescaled modulation space defined by the norm
\eqq{\norm{f}{M^\rho _A}:=\sum _{\xi \in A\Bo{Z}}\norm{\hhat{f}}{L^\rho (\xi +I_A)},\qquad I_A:=[ -\tfrac{A}{2},\tfrac{A}{2}) .}
It is easy to see that $M_A^\rho $ is a Banach algebra with a product estimate:
\eqq{\norm{fg}{M_A^\rho}\le CA^{1-\frac{1}{\rho}}\norm{f}{M^\rho _A}\norm{g}{M^\rho _A}.}
Mimicking the proof of Lemma~\ref{lem:U_k}, we see that the operators $U_k$ defined as in Definition~\ref{defn:U_k} satisfy
\eq{est:A1}{\norm{U_k[\phi ](t)}{M_A^\rho}\le t^{\frac{k-1}{2}}\big( CA^{1-\frac{1}{\rho}}\tnorm{\phi}{M_A^\rho}\big) ^{k-1}\tnorm{\phi}{M_A^\rho},\qquad t\ge 0,\quad k\ge 1.}
We also recall that from Corollary~\ref{cor:lwp}, the power series expansion of the solution map $u[\phi ]=\sum _{k\ge 1}U_k[\phi ]$ is verified in $C([0,T];M_A^2)$ whenever
\eq{cond:A1}{0<T\ll \big( A^{\frac{1}{2}}\tnorm{\phi}{M_A^2}\big) ^{-2}.}

For the proof of norm inflation in $D^{[\al ]}_{p,q}$, we restrict the initial data $\phi$ to those of the form \eqref{cond:phi}; for given $N\gg 1$, we set
\eqq{\hhat{\phi}:=rA^{-\frac{1}{p}}N^{\frac{1}{p}}\chi_{(N+I_A)\cup (2N+I_A)},}
where $r>0$ and $1\ll A\ll N$ will be specified later according to $N$.
Then, since $\tnorm{\phi}{M_A^2}\sim rA^{\frac{1}{2}-\frac{1}{p}}N^{\frac{1}{p}}$, the condition \eqref{cond:A1} is equivalent to
\eq{cond:A2}{0<r(TN^2)^{\frac{1}{2}}\Big( \frac{A}{N}\Big) ^{1-\frac{1}{p}}\ll 1.}
Moreover, it holds that
\eq{est:A2}{\norm{U_1[\phi ](T)}{D^{[\al ]}_{p,q}}=\norm{\phi}{D^{[\al ]}_{p,q}}\sim r(\log N)^\al ,\qquad T\ge 0,}
and similarly to Lemma~\ref{lem:U_p} (i), that
\eq{est:A3}{\norm{U_3[\phi ](T)}{D^{[\al ]}_{p,q}}&\ge cT\big( rA^{-\frac{1}{p}}N^{\frac{1}{p}}\big) ^3A^2\norm{\F ^{-1}\chi _{I_{A/2}}}{D^{[\al ]}_{p,q}},\qquad 0<T\le \tfrac{1}{100}N^{-2},\\
&=c\Big[ r(TN^2)^{\frac{1}{2}}\Big( \frac{A}{N}\Big) ^{1-\frac{1}{p}}\Big] ^2r\Big( \frac{A}{N}\Big) ^{-\frac{1}{p}}f_{p,q}^\al (A),}
where
\eqq{f_{p,q}^\al (A):=\norm{\F ^{-1}\chi _{I_{A/2}}}{D^{[\al ]}_{p,q}}\sim \begin{cases}
(\log A)^{\al +\frac{1}{q}}, &\al >-\frac{1}{q},\\
(\log \log A)^{\frac{1}{q}}, &\al =-\frac{1}{q},\\
~1, &\al <-\frac{1}{q}.\end{cases}}

For estimating $U_{2l+1}[\phi]$, $l\ge 2$ in $D^{[\al ]}_{p,q}$, we first observe that
\eqq{\norm{U_k[\phi ](T)}{D^{[\al ]}_{p,q}}\le \norm{\F ^{-1}\chi _{\supp{\hhat{U_k[\phi ]}(T)}}}{D^{[\al ]}_{p,q}}\norm{\hhat{U_k[\phi ]}(T)}{L^\I}. 
}
A simple computation yields that
\eqq{\norm{\F ^{-1}\chi _\Omega}{D^{[\al ]}_{p,q}}\le C\norm{\F ^{-1}\chi _{I_{|\Omega |}}}{D^{[\al ]}_{p,q}}}
for any measurable set $\Omega \subset \R$ of finite measure.
From Lemma~\ref{lem:supp}, we have
\eqq{\big| \supp{\hhat{U_k[\phi ]}(T)}\big| \le C^kA, \qquad T\ge 0,\quad k\ge 1,}
and hence,
\eqq{\norm{\F ^{-1}\chi _{\supp{\hhat{U_k[\phi ]}(T)}}}{D^{[\al ]}_{p,q}}\le C\norm{\F ^{-1}\chi _{I_{C^kA}}}{D^{[\al ]}_{p,q}}\le C^kf_{p,q}^\al (A).}
Moreover, similarly to Lemma~\ref{lem:U_k_H^s} (ii), we use Young's inequality, \eqref{est:A1} and Lemma~\ref{lem:a_k} to obtain
\eqq{\norm{\hhat{U_k[\phi ]}(T)}{L^\I}&\le \sum _{\mat{k_1,k_2,k_3\ge 1\\k_1+k_2+k_3=k}}\int _0^T\norm{\hhat{U_{k_1}[\phi ]}(t)}{M^{\frac{3}{2}}_A}\norm{\hhat{U_{k_2}[\phi ]}(t)}{M^{\frac{3}{2}}_A}\norm{\hhat{U_{k_3}[\phi ]}(t)}{M^{\frac{3}{2}}_A}\,dt\\
&\le \int _0^Tt^{\frac{k-3}{2}}\,dt\cdot \big( CrA^{1-\frac{1}{p}}N^{\frac{1}{p}}\big) ^{k-3}\big( CrA^{\frac{2}{3}-\frac{1}{p}}N^{\frac{1}{p}}\big) ^{3}\\
&\le C\big( CrT^{\frac{1}{2}}A^{1-\frac{1}{p}}N^{\frac{1}{p}}\big) ^{k-1}rA^{-\frac{1}{p}}N^{\frac{1}{p}},\qquad T\ge 0,\quad k\ge 3.
}
Hence, we have
\eq{est:A4}{\norm{U_k[\phi ](T)}{D^{[\al ]}_{p,q}}\le C\Big[ Cr(TN^2)^{\frac{1}{2}}\Big( \frac{A}{N}\Big) ^{1-\frac{1}{p}}\Big] ^{k-1}r\Big( \frac{A}{N}\Big) ^{-\frac{1}{p}}f_{p,q}^\al (A),\quad
T\ge 0,~~k\ge 3.}

From \eqref{cond:A2}--\eqref{est:A4}, we only need to check if there exist $r,A,T$ such that
\eq{cond:A3}{1\ll A\ll N,\qquad r\ll (\log N)^{-\al} ,\qquad (TN^2)\le \tfrac{1}{100},\qquad\qquad \\
\Big[ r(TN^2)^{\frac{1}{2}}\Big( \frac{A}{N}\Big) ^{1-\frac{1}{p}}\Big] ^2\ll 1 \ll \Big[ r(TN^2)^{\frac{1}{2}}\Big( \frac{A}{N}\Big) ^{1-\frac{1}{p}}\Big] ^2r\Big( \frac{A}{N}\Big) ^{-\frac{1}{p}}f^\al _{p,q}(A).}
When $1\le p<\frac{3}{2}$, it holds that $2(1-\frac{1}{p})\ge 0>2(1-\frac{1}{p})-\frac{1}{p}$.
Hence, we may choose
\eqs{r=(\log N)^{\min \{ -\al ,0\} -1},\quad A=N^{\frac{1}{2}},\quad T=\tfrac{1}{100}N^{-2},}
which clearly satisfies \eqref{cond:A3}.
(Note that $f^\al _{p,q}(A)\gec 1$ for any $p,q,\al$.)

If $\frac{3}{2}\le p<\I$, \eqref{cond:A3} would imply that
\eqq{1 \ll \Big[ r(TN^2)^{\frac{1}{2}}\Big( \frac{A}{N}\Big) ^{1-\frac{1}{p}}\Big] ^2r\Big( \frac{A}{N}\Big) ^{-\frac{1}{p}}f^\al _{p,q}(A)\lec r^3f^\al _{p,q}(A)\ll (\log N)^{-3\al}f^\al _{p,q}(A).}
In particular, when $\al >-\frac{1}{q}$ this condition requires
\eqq{(\log N)^{3\al}\ll (\log N)^{\al +\frac{1}{q}},}
which shows the necessity of the restriction $\al <\frac{1}{2q}$ in our argument.
We now see the possibility of choosing $r,A,T$ with the condition \eqref{cond:A3} in the following two cases separately: (a) If $1\le q<\I$ and $0\le \al<\frac{1}{2q}$, we may take for instance
\eqq{r=(\log N)^{-\al}(\log \log N)^{-1},\quad A=N(\log \log N)^{-1},\quad T=\tfrac{1}{100}N^{-2}.}
(Note that $f^\al _{p,q}(A)\sim f^\al _{p,q}(N)\sim (\log N)^{\al +\frac{1}{q}}$.)
(b) If $\al <0$, we take
\eqq{r=(\log N)^{-\al}(\log \log N)^{-1},\quad A=N(\log N)^{\al (1-\frac{1}{p})^{-1}},\quad T=\tfrac{1}{100}N^{-2}.}
In both cases we easily show \eqref{cond:A3}.

Finally, we assume $1\le p<\frac{3}{2}$ and prove norm inflation in $D^s_{p,q}$ for $s<-\frac{2}{3}$.
We use the initial data $\phi$ of the form
\eqq{\hhat{\phi}:=rN^{-s}\chi_{[N,N+1]\cup [2N,2N+1]}.}
Then, the condition \eqref{cond:A1} with $A=1$ is equivalent to
\eqq{0<T^{\frac{1}{2}}rN^{-s}=(TN^2)^{\frac{1}{2}}rN^{-s-1}\ll 1.}
Repeating the argument above we also verify that
\eqs{\norm{U_1[\phi ](T)}{D^s_{p,q}}=\norm{\phi}{D^s_{p,q}}\sim r,\\
\norm{U_3[\phi ](T)}{D^s_{p,q}}\ge c\big( T^\frac{1}{2}rN^{-s}\big) ^2rN^{-s}=c(TN^2)r^3N^{-3s-2}\qquad \text{if $T\le \tfrac{1}{100}N^{-2}$},\\
\norm{U_k[\phi ](T)}{D^s_{p,q}}\le C\big( CT^\frac{1}{2}rN^{-s}\big) ^{k-1}rN^{-s},\qquad T\ge 0,~k\ge 3.
}
Hence, we set
\eqq{r=N^{s+\frac{2}{3}}\log N,\qquad T=\tfrac{1}{100}N^{-2},}
so that for $s<-\frac{2}{3}$ we have
\eqs{\norm{U_1[\phi ](T)}{D^s_{p,q}}\sim N^{s+\frac{2}{3}}\log N\ll 1,\qquad \norm{U_3[\phi ](T)}{D^s_{p,q}}\gec (\log N)^3\gg 1, \\
\sum _{l\ge 2}\norm{U_{2l+1}[\phi ](T)}{D^s_{p,q}}\lec N^{-\frac{2}{3}}(\log N)^5\ll 1,}
from which norm inflation is detected by letting $N\to \I$.
\end{proof}

\bigskip
\medskip
\noindent \textbf{Acknowledgments:}
The author would like to thank Tadahiro Oh for his generous suggestion and encouragement.
This work is partially supported by JSPS KAKENHI Grant-in-Aid for Young Researchers (B)
No.~24740086 and No.~16K17626.

%%%%%%%%%%%%%%%%%%%%%%%%%%%%%%%%%%%
%%%%%%%%%%%%%%%%%%%%%%%%%%%%%%%%%%%
%%%%%%%%%%%%%%%%%%%%%%%%%%%%%%%%%%%

\bigskip
\bigskip
% \medskip

\bigskip
% \bigskip


\begin{thebibliography}{00}
\bibitem{AC09} T.~Alazard and R.~Carles, \emph{Loss of regularity for supercritical nonlinear Schr\"odinger equations}, Math.~Ann. \textbf{343} (2009), no. 2, 397--420.
\bibitem{BT06} I.~Bejenaru and T.~Tao, \emph{Sharp well-posedness and ill-posedness results for a quadratic non-linear Schr\"odinger equation}, J.~Funct.~Anal. \textbf{233} (2006), no. 1, 228--259.
The latest version is in \texttt{arXiv:math/0508210}
% \bibitem{BO09} \'A.~B\'enyi and K.A.~Okoudjou, \emph{Local well-posedness of nonlinear dispersive equations on modulation spaces}, Bull.~Lond.~Math.~Soc. \textbf{41} (2009), no. 3, 549--558.
\bibitem{B93-1} J.~Bourgain, \emph{Fourier transform restriction phenomena for certain lattice subsets and applications to nonlinear evolution equations, I, Schr\"odinger equations}, Geom. Funct. Anal. \textbf{3} (1993), 107--156.
% \bibitem{B93-2} J. Bourgain, \textit{Fourier transform restriction phenomena for certain lattice subsets and applications to nonlinear evolution equations, II, the KdV-equation}, Geom. Funct. Anal. \textbf{3} (1993), 209--262.
% \bibitem{B97} J. Bourgain, \textit{Periodic Korteweg de Vries equation with measures as initial data}, Selecta Math. (N.S.) \textbf{3} (1997), no. 2, 115--159.
\bibitem{BGT02} N.~Burq, P.~G\'erard, and N.~Tzvetkov, \emph{An instability property of the nonlinear Schr\"odinger equation on $S^d$}, Math.~Res.~Lett. \textbf{9} (2002), no. 2-3, 323--335.
\bibitem{C07} R.~Carles, \emph{Geometric optics and instability for semi-classical Schr\"odinger equations}, Arch.~Ration.~Mech.~Anal. \textbf{183} (2007), no. 3, 525--553.
\bibitem{CDS12} R.~Carles, E.~Dumas, and C.~Sparber, \emph{Geometric optics and instability for NLS and Davey-Stewartson models}, J.~Eur.~Math.~Soc. \textbf{14} (2012), no. 6, 1885--1921.
\bibitem{CK17} R.~Carles and T.~Kappeler, \emph{Norm-inflation with infinite loss of regularity for periodic NLS equations in negative Sobolev spaces}, Bull. Soc. Math. France \textbf{145} (2017), no. 4, 623--642.
\bibitem{CP16} A.~Choffrut and O.~Pocovnicu, \emph{Ill-posedness of the cubic nonlinear half-wave equation and other fractional NLS on the real line}, Int.~Math.~Res.~Not. IMRN (\textbf{2018}), no. 3, 699--738.
\bibitem{CCT03} M.~Christ, J.~Colliander, and T.~Tao, \emph{Asymptotics, frequency modulation, and low regularity ill-posedness for canonical defocusing equations}, Amer.~J.~Math. \textbf{125} (2003), no. 6, 1235--1293.
\bibitem{CCT03p-1} M.~Christ, J.~Colliander, and T.~Tao, \emph{Ill-posedness for nonlinear Schr\"odinger and wave equations}, preprint (2003). \texttt{arXiv:math/0311048}
\bibitem{CCT03p-2} M.~Christ, J.~Colliander, and T.~Tao, \emph{Instability of the periodic nonlinear Schr\"odinger equation}, preprint (2003). \texttt{arXiv:math/0311227}
\bibitem{CCT08} M.~Christ, J.~Colliander, and T.~Tao, \emph{A priori bounds and weak solutions for the nonlinear Schr\"odinger equation in Sobolev spaces of negative order}, J.~Funct.~Anal. \textbf{254} (2008),  no. 2, 368--395.
\bibitem{FLS87} F.~Falk, E.W.~Laedke, and K.H.~Spatschek, \emph{Stability of solitary-wave pulses in shape-memory alloys}, Phys.~Rev.~B \textbf{36} (1987), no. 6, 3031--3041.
\bibitem{F83} H.G.~Feichtinger, \emph{Modulation spaces on locally compact Abelian groups}, Technical Report,
University of Vienna, 1983; Published in ``Proc.~Internat.~Conf.~on Wavelets and Applications'', New Delhi Allied Publishers, 2003, 1--56.
\bibitem{G00p} A.~Gr\"unrock, \emph{Some local wellposedness results for nonlinear Schr\"odinger equations below $L^2$}, preprint (2000). \texttt{arXiv:math/0011157}
\bibitem{G17} S.~Guo, \emph{On the 1D cubic nonlinear Schr\"odinger equation in an almost critical space}, J.~Fourier Anal.~Appl. \textbf{23} (2017), no. 1, 91--124.
\bibitem{GO18} Z.~Guo and T.~Oh, \emph{Non-existence of solutions for the periodic cubic NLS below $L^2$}, Int. Math. Res. Not. IMRN (\textbf{2018}), no. 6, 1656--1729.
\bibitem{GNT09} S.~Gustafson, K.~Nakanishi, and T.P.~Tsai, \emph{Scattering theory for the Gross-Pitaevskii equation in three dimensions}, Commun.~Contemp.~Math. \textbf{11} (2009), no. 4, 657--707.
\bibitem{HMO16} H.~Huh, S.~Machihara, and M.~Okamoto, \emph{Well-posedness and ill-posedness of the Cauchy problem for the generalized Thirring model}, Differential Integral Equations \textbf{29} (2016), no. 5-6, 401--420.
\bibitem{IO15} T.~Iwabuchi and T.~Ogawa, \emph{Ill-posedness for the nonlinear Schr\"odinger equation with quadratic non-linearity in low dimensions}, Trans.~Amer.~Math.~Soc. \textbf{367} (2015), no. 4, 2613--2630.
\bibitem{IU15} T.~Iwabuchi and K.~Uriya, \emph{Ill-posedness for the quadratic nonlinear Schr\"odinger equation with nonlinearity $|u|^2$}, Commun.~Pure Appl.~Anal. \textbf{14} (2015), no. 4, 1395--1405.
\bibitem{KPV96-NLS} C.E.~Kenig, G.~Ponce, and L.~Vega, \emph{Quadratic forms for the $1$-D semilinear Schr\"odinger equation}, Trans.~Amer.~Math.~Soc. \textbf{348} (1996), no. 8, 3323--3353.
\bibitem{KPV01} C.E.~Kenig, G.~Ponce, and L.~Vega, \emph{On the ill-posedness of some canonical dispersive equations}, Duke Math.~J. \textbf{106} (2001), no. 3, 617--633. 
\bibitem{KVZ17p} R.~Killip, M.~Vi\c{s}an, and X.~Zhang, \emph{Low regularity conservation laws for integrable PDE}, preprint (2017). \texttt{arXiv:1708.05362}
\bibitem{K09} N.~Kishimoto, \emph{Low-regularity bilinear estimates for a quadratic nonlinear Schr\"odinger equation}, J.~Differential Equations \textbf{247} (2009), no. 5, 1397--1439.
\bibitem{KT10} N.~Kishimoto and K.~Tsugawa, \emph{Local well-posedness for quadratic nonlinear Schr\"odinger equations and the ``good'' Boussinesq equation}, Differential Integral Equations \textbf{23} (2010), no. 5-6, 463--493.
\bibitem{KT07} H.~Koch and D.~Tataru, \emph{A priori bounds for the 1D cubic NLS in negative Sobolev spaces}, Int.~Math.~Res.~Not. IMRN \textbf{2007}, no. 16, Art.ID rnm053, 36 pp.
\bibitem{KT12} H.~Koch and D.~Tataru, \emph{Energy and local energy bounds for the 1-d cubic NLS equation in $H^{-1/4}$}, Ann.~Inst.~H.~Poincar\'e Anal.~Non Lin\'eaire \textbf{29} (2012), no. 6, 955--988.
\bibitem{KT16p} H.~Koch and D.~Tataru, \emph{Conserved energies for the cubic NLS in 1-d}, preprint (2016). \texttt{arXiv:1607.02534}
\bibitem{MO15} S.~Machihara and M.~Okamoto, \emph{Ill-posedness of the Cauchy problem for the Chern-Simons-Dirac system in one dimension}, J.~Differential Equations \textbf{258} (2015), no. 4, 1356--1394.
\bibitem{MO16} S.~Machihara and M.~Okamoto, \emph{Sharp well-posedness and ill-posedness for the Chern-Simons-Dirac system in one dimension}, Int.~Math.~Res.~Not. (\textbf{2016}), no. 6, 1640--1694.
\bibitem{M09} L.~Molinet, \emph{On ill-posedness for the one-dimensional periodic cubic Schr\"odinger equation}, Math.~Res.~Lett. \textbf{16} (2009), no. 1, 111--120.
\bibitem{O17} T.~Oh, \emph{A remark on norm inflation with general initial data for the cubic nonlinear Schr\"odinger equations in negative Sobolev spaces}, Funkcial.~Ekvac. \textbf{60} (2017), 259--277. 
\bibitem{OS12} T.~Oh and C.~Sulem, \emph{On the one-dimensional cubic nonlinear Schr\"odinger equation below $L^2$}, Kyoto J.~Math. \textbf{52} (2012), no. 1, 99--115.
\bibitem{OW15p} T.~Oh and Y.~Wang, \emph{On the ill-posedness of the cubic nonlinear Schr\"odinger equation on the circle}, to appear in An. \c{S}tiin\c{t}. Univ. Al. I. Cuza Ia\c{s}i. Mat. (N.S.).
\bibitem{OW18p} T.~Oh and Y.~Wang, \emph{Global well-posedness of the one-dimensional cubic nonlinear Schr\"odinger equation in almost critical spaces}, preprint (2018). \texttt{arXiv:1806.08761}
\bibitem{Ok17} M.~Okamoto, \emph{Norm inflation for the generalized Boussinesq and Kawahara equations}, Nonlinear Anal. \textbf{157} (2017), 44--61.
\bibitem{RSW12} M.~Ruzhansky, M.~Sugimoto, and B.~Wang, \emph{Modulation spaces and nonlinear evolution equations}, Evolution equations of hyperbolic and Schr\"odinger type, 267--283, Progr.~Math., \textbf{301}, Birkh\"auser/Springer Basel AG, Basel, 2012.
\bibitem{T87} Y.~Tsutsumi, \emph{$L^2$-solutions for nonlinear Schr\"odinger equations and nonlinear groups}, Funkcial.~Ekvac. \textbf{30} (1987), no. 1, 115--125.
\end{thebibliography}
\end{document}